\documentclass[12pt,a4paper]{amsart}
\usepackage[legalpaper,margin=2cm]{geometry}
\usepackage[utf8]{inputenc}
\usepackage{amsfonts, amssymb, euscript,graphics}
\usepackage{amsmath}     %
\usepackage{amscd}
\usepackage{graphicx}
\usepackage{eepic}
\usepackage{comment}
\usepackage{psfrag}
\usepackage{xcolor}      %
\usepackage[inline]{enumitem} %
\usepackage{colonequals} %
\usepackage{thm-restate} %
\usepackage{lmodern}     %
\usepackage{float}       %
\usepackage{varioref}    %
\usepackage{hyperref}
\usepackage{wrapfig}     %
\usepackage{caption}

\usepackage{tikz-cd}  
\tikzset{                %
  symbol/.style={        %
    draw=none,
    every to/.append style={ 
      edge node={node [sloped, allow upside down, auto=false]{$#1$}}}
  }
}

\usepackage{xfrac}       %
\UseCollection{xfrac}{plainmath}

\usepackage{cleveref}    %
\crefname{subsection}{subsection}  {subsections} %
\crefname{prop}      {proposition} {propositions}
\crefname{con}       {construction}{constructions}
\crefname{rem}       {remark}      {remarks}
\crefname{ex}        {example}     {examples}
\crefname{cor}       {corollary}   {corollary}

\usepackage
[style=alphabetic, sorting=ynt]{biblatex}
\theoremstyle{definition}
\declaretheorem[name=Theorem, parent =section]{thm}
\declaretheorem[name=Lemma,   sibling=thm    ]{lem}
\newtheorem{prop} [thm]{Proposition}
\newtheorem{defin}[thm]{Definition}
\newtheorem{rem}  [thm]{Remark}
\newtheorem{con}  [thm]{Construction}

\newtheorem{cor}  [thm]{Corollary}

\newcommand{\R}    {\mathbb{R}}  %
\newcommand{\F}    {\mathbb{F}}  %
\newcommand{\Z}    {\mathbb{Z}}  %
\newcommand{\Q}    {\mathbb{Q}}  %
\newcommand{\Ll}   {\mathcal{T}} %
\let\P\relax
\newcommand{\P}    {\mathbb{P}}  %
\let\phi\relax
\newcommand{\phi}  {\varphi}     %
\let\kappa\relax
\newcommand{\kappa}{\varkappa}   %
\let\d\relax
\newcommand{\d}    {\partial}    %
\newcommand{\B}    {B}           %
\newcommand{\C}    {\mathcal{C}} %
\let\H\relax
\newcommand{\H}    {\mathsf{H}}  %
\newcommand{\Zy}   {Z}           %
\newcommand{\eps}  {\varepsilon} %
\newcommand{\M}    {\mathcal{M}} %
\newcommand{\X}    {\mathsf{X}}  %
\newcommand{\nC}   {\mathsf{C}}  %
\let\O\relax
\newcommand{\O}    {\mathsf{O}}  %

\DeclareMathOperator{\im}{Im}
\let\ker\relax
\DeclareMathOperator{\ker}{Ker}
\DeclareMathOperator{\coker}{Coker}
\DeclareMathOperator{\rk}{rk}
\DeclareMathOperator{\Ann}{Ann}
\DeclareMathOperator{\Uni}{Uni}
\DeclareMathOperator{\Cr}{Cr}
\DeclareMathOperator{\ch}{char}
\DeclareMathOperator{\Tors}{Tors}
\newcommand{\id}{\mbox{id}}           %
\newcommand{\Ht}{\mathrm{ht}}         %
\newcommand{\cf}{\mathrm{cf}}         %
\newcommand{\choicemantra}{NB: this construction involves some choices.}
\newcommand{\Csing}{\C_{\text{sing}}} %
\newcommand{\dM}{\d M}                %

\let\leq\relax
\newcommand{\leq}{\leqslant}            %
\let\geq\relax
\newcommand{\geq}{\geqslant}            %
\newcommand{\xto}{\xrightarrow}         %
\newcommand{\into}{\hookrightarrow}     %
\newcommand{\onto}{\twoheadrightarrow}  %
\let\emptyset\relax
\newcommand{\emptyset}{\varnothing}
\newcommand{\<}{\langle}                %
\let\>\relax
\newcommand{\>}{\rangle}                %
\newcommand{\triright}{\noindent $\vartriangleright$ } %
\newcommand{\trileft }{\hfill    $\vartriangleleft$  } %
\newcommand{\she}{\stackrel{s}{\simeq}} %
\let\tilde\relax
\newcommand{\tilde}{\widetilde}         %
\newcommand{\congrot}{\rotatebox[origin=c]{-20}{$\cong$}} %

\newcommand{\subs}[1]
   {\refstepcounter{subsection}
   \medskip\noindent
   {\it\arabic{section}.\arabic{subsection}.\hspace{.25em}\ignorespaces#1.}
   \ignorespaces}

\title{Enhanced Bruhat decomposition and Morse theory}

\author{Petr Pushkar}
\address{\noindent Department of Mathematics \newline National Research University Higher School
  of Economics \newline Moscow,  Russian Federation}
\email{petya.pushkar@gmail.com}

\author{Misha Tyomkin}
\address{Department of Mathematics \& International Laboratory of Algebraic
  Topology and Its Applications \newline
  National Research University Higher School
  of Economics \newline Moscow, Russian Federation}
\email{mstyomkin@edu.hse.ru}

\date{}

\begin{document}

\begin{abstract}
  Morse function is called strong if all its critical values are pairwise
  distinct. Given such a function $f$ and a field $\F$ Barannikov constructed a
  pairing of some of the critical points \mbox{of $f$}, which is now also known
  as barcode. With every Barannikov pair we naturally associate (up to sign) an
  element of $\F \setminus \{0\}$; we call it Bruhat number. The paper is
  devoted to the study of these Bruhat numbers. We investigate several
  situations where the product of all these numbers (some being raised to the
  power $-1$) is independent of $f$ and interpret it as a Reidemeister torsion.
  We apply our results in the setting of one-parameter Morse theory by proving
  that generic path of functions must satisfy a certain equation mod 2 (this was
  initially proven in \cite{Akhm} under additional assumptions).

  On the linear-algebraic level our constructions are served by the following
  variation of a classical Bruhat decomposition for $GL(\F)$. A unitriangular
  matrix is an upper triangular one with 1's on the diagonal. Consider all
  rectangular matrices over $\F$ up to left and right multiplication by
  unitriangular ones. Enhanced Bruhat decomposition describes canonical
  representative in each equivalence class.
\end{abstract}

\maketitle

\section*{Introduction}

\subs{The context} In this paper we study a certain invariant of a strong Morse
function on a smooth manifold (which is supposed to be closed most of the time).
Recall that a function is called Morse if all its critical points are
non-degenerate. The function is called strong if all its critical values are
pairwise distinct. Morse theory is a classical branch of differential topology:
one can extract a lot of topological information about the manifold in terms of
a Morse function. On the other hand, Morse functions arise naturally in various
situations and their properties are interesting in their own right.

Recall that each critical point of a Morse function carries a number~--- its
index. The very first theorem of Morse theory states that one can find a
CW-complex homotopy equivalent to the manifold, which $k$-cells correspond to
critical points of index $k$.

Suppose now that one is given not only a strong Morse function $f$ on a manifold
$M$, but also a field $\F$. To such a data Barannikov \cite{Bar} associated a
powerful combinatorial structure on the set of critical points of $f$. Nowadays
it is also known as barcode and serves as a well-established tool in applied and
symplectic topology. This structure is a pairing of some critical points of
neighboring indices. This pairing does depend on the field $\F$. For example,
the number of unpaired critical points of index $k$ equals to
$\dim_\F \H_k(M; \F)$. Moreover, this pairing relies crucially on the fact the
function is strong, i.e. critical points are linearly ordered. If one starts to
change a Morse function so that along the way it fails to be strong (i.e. two
critical values collide), the pairing changes. We call these pairs Barannikov
pairs. We will sketch their definition in the next subsection.

\subs{Bruhat numbers of a single function}
\label{subs:morse}
To state our results consicely we prefer to speak about oriented strong Morse
functions. Roughly, the Morse function $f$ is called oriented if one has chosen
orientation on all the cells in the CW-complex constructed from it. This
condition is both technical and minor: one can always orient a Morse function.
As usual in topology, this choice only alters certain signs.

Our first result is not a theorem, but rather a construction. Namely with any
given oriented Morse function $f$ on a manifold $M$ and a field $\F$ we
naturally associate a non-zero number with each Barannikov pair. Here by number
we mean an element of $\F^* = \F \setminus \{0\}$. We call these numbers
``Bruhat numbers'' of $f$ over $\F$. The reason is that their construction is
closely related to the classical Bruhat decompoistion for $GL(\F)$. This paper
is devoted to the study of these numbers from different perspectives. The first
thing to mention is that changing the orientation of $f$ may only change signs
of some Bruhat numbers.

We will now sketch the contruction of both Barannikov pairs and Bruhat numbers.
Recall that if one chooses a generic Riemannian metric $g$ on $M$ then they can
consider a Morse complex which counts homology of $M$ (this is a classical
construction, it has nothing to do with a field). It is a chain complex of free
abelian groups, formally spanned by critical points (points of index $k$ are of
degree $k$). Thus the Morse differential is nothing but an integer matrix:
differential of a critical point $x$ is a linear combination of points of
smaller index. The coefficient of the point $y$ in this linear combination is
the number of antigradient flowlines from $x$ to $y$, counted with appropriate
signs (thus it is non-zero only if $f(x) > f(y)$). Since the function is strong,
the set of critical points is ordered by increasing of critical values. Next, we
note that choosing a different metric $g'$ results in a different matrix of
Morse differential. More precisely, these two matrices differ by a conjugation
by unitriangular (i.e. triangular with 1's on the diagonal) one. We treat this
unitriangular matrix as that of a change of basis of a complex. As a
recollection, the class of a matrix of Morse differential up to conjugation by a
unitriangular matrix is a well-defined invariant of $f$ (i.e. doesn't depend on
a metric). The corresponding classification problem is, however, very hard, so
we consider the coefficients in $\F$. In such a case we prove that every complex
is isomorphic, up to unitriangular change of basis, to the direct sum of
$0 \to \F \xto{\times \lambda} \F \to 0$ and $0 \to \F \to 0$. Generators of the
former (which are themselves critical points) are Barannikov pairs. The
corresponding number $\lambda$ is a Bruhat number on a pair.

Weak Morse inequalities state that the number of critical points of index $k$ is
greater or equal \mbox{to $\dim \H_k(M; \Q)$}. Let $\# \Cr(f)$ be the overall
number of critical points of $f$. It is easy to show that if
$\# \Cr(f) = \sum_k \dim \H_k(M; \Q)$ then the Morse differential (w.r.t. any
metric) must be identically zero. The next statement is applicable when this is
not the case.

\begin{restatable}{cor}{CorShortMatrRestate}
  \label{cor:short_matr}   
  Let $f$ be an oriented strong Morse function on $M$. Suppose \mbox{that
    $\# \Cr(f) > \sum_k \dim \H_k(M; \Q)$.} Then one can find two critical
  points $x$ and $y$ of neighboring indices s.t. the number of antigradient
  flowlines from $x$ to $y$, counted with appropriate signs, is the same for any
  Riemannian metric. This number is non-zero and equals to some Bruhat number of
  $f$ over $\Q$.
\end{restatable}

In \Cref{subs:morse_comp} this statement is derived from \Cref{thm:matr} which
says that the matrix of a Morse differential of $f$ w.r.t. to any metric must
obey certain restrictions. These restrictions, in turn, are expressed in terms
of Bruhat numbers and Barannikov pairs. They are in the spirit of Bruhat
decomposition, see the mentioned subsection for the precise statement and
example. Note that in particular we claim that at least one Bruhat number over
$\Q$ must be integer. For a general Bruhat number this is false, however.

\subs{Interplay with the theory of torsions}
\label{subs:br_tors}
In \Cref{prop:realiz} we prove that if $\F$ is either $\Q$ or $\F_p$ then one
can find a Morse function which has any prescribed number $\lambda \in \F^*$ as
one of its Bruhat numbers; the manifold $M$ may be any with $\dim M \geq 4$.
Thus one has no control over individual Bruhat number~--- it may turn out to be
any number. We propose, however, to consider the alternating product of all the
Bruhat numbers. The word ``alternating'' here means that each Bruhat number is
raised to the power $\pm 1$ depending on the parity of indices of critical
points involved in the corresponding Barannikov pair. The term is used in
analogy of Euler characteristic, which the alternating sum of, say, cells in the
CW-complex. In the following statement this product is considered up to sign.

\begin{restatable}{thm}{ThmNoHomRestate}
  \label{thm:no_hom}
  Let $f$ be a strong Morse function on $M$ and $\F$ be a field. Suppose that
  $\H_k(M) = 0$ for all $0 < k < \dim M$. Then the alternating product of all
  Bruhat numbers (as an element from $\sfrac{\F^*}{\pm 1}$) is independent of
  $f$.
\end{restatable}

This is discussed in \Cref{subs:tors,subs:no_hom}; in particular we interpret
this alternating product as a certain kind torsion. For example, if $M = \R\P^n$
and $\F = \Q$ then this product equals to $\pm 2^{[n/2]}$, where brackets denote
the integral part.

We shall now quickly recall the notion of Reidemeister torsion of a manifold.
This is a purely algebro-topological invariant which itself has nothing to do
with Morse theory. Let $\Z[\pi]$ denote the integral group ring of a fundamental
group $\pi = \pi_1(M)$. Suppose one is given a map of rings
$\rho : \Z[\pi] \to \F$ for some field $\F$. Then they may considered homology
of $M$ with coefficients twisted by $\rho$. If it vanishes, then they may
furthermore define the Reidemeister torsion of $M$ w.r.t. $\rho$, which is an
element from the quotient group $\sfrac{\F^*}{\rho(\pm \pi)}$. This invariant is
useful for distinguishing homotopy equivalent but non-homeomorphic manifolds,
such as lens spaces.

We will now pour Morse theory into this setting. Suppose that now we are given
not only oriented strong Morse function $f$ and a field $\F$ but also a map
$\rho : \Z[\pi] \to \F$. In this case we construct twisted Barannikov pairs and
Bruhat numbers. In general, the alternating product of twisted Bruhat numbers
may well depend on $\F$. The interesting case, however, is when one is able to
define the Reidemeister torsion, i.e. when twisted homology vanishes.

\begin{restatable}{thm}{ThmRTRestate}
  \label{thm:rt}
  Let $f$ be an oriented strong Morse function on a manifold $M$, $\F$ be a
  field and $\rho \colon \Z[\pi] \to \F$ be a homomorphism of rings. Suppose
  that twisted homology vanishes. Then the alternating product of twisted Bruhat
  numbers of $f$ equals to the Reidemeister torsion of $M$. In particular, it is
  independent of $f$.
\end{restatable}

\subs{One-parameter Morse theory}
\label{subs:one_par}
Let us now consider not a single strong Morse function but a generic path in the
space of all functions on $M$. All but finitely points on this path are
themselves strong Morse functions. However, after passing a function which fails
to be strong and Morse Barannikov pairs and Bruhat numbers change. Thanks to the
genericity assumption on the path it is possible to describe exactly the list of
possible changes. For Barannikov pairs this was done already in \cite{Bar} (see
also \cite{CSEM} and pictures in the survey \cite{EH}). We do the same for
Bruhat numbers on them. In particular, this gives a proof of \Cref{thm:no_hom}.
Moreover, it enables us to prove the theorem of Akhmetev-Cencelj-Repovs
\cite{Akhm} in greater generility, which we shall now describe.

Let $\{ f_t \}$ be a generic path on functions on $M$, i.e. $f_t$ is a function
from $M$ to $\R$ for each $t \in [-1,1]$. To such a path one associates a Cerf
diagram. It is a subset of $[-1, 1] \times \R$ consisting of points $(t, x)$
s.t. $x$ is a critical value of $f_t$. Practically, it is a set of plane arcs
which start and end either at cusps or at vertical lines $t = \pm 1$. The only
possible singularities of a Cerf diagram are cusps and simple transversal
self-intersections. By orienting all the functions in a path one may associate a
sign with each cusp. The parity of negative cusps is a well-defined invariant of
a path $\{ f_t \}$, i.e. it doesn't depend on the orientations. Another
invariant of a path is a number of self-intersections of its Cerf diagram (i.e.
the number points $t_0$ s.t. $f_{t_0}$ is not strong). In the following
statement we consider functions on a cylinder $M \times [0,1]$, which is
manifold with boundary. Morse theory translates to this setting with ease.
    
\begin{restatable}{cor}{CorAkhCorRestate}
  \label{cor:akh_cor}
  Let $\{ f_t \}$ be a generic path of functions on a cylinder $M \times [0,1]$
  s.t. both $f_{-1}$ and $f_1$ have no critical points. Let $\X$ be the number
  of self-intersections of its Cerf diagram and $\nC$ be the number of negative
  cusps. Then one has
  \[
    \X + \nC = 0 \pmod{2}.
  \]
\end{restatable}

In \cite{Akhm} this statement was proved under additional assumptions on $M$. In
\Cref{subs:akh} we derive this corollary from a more general \Cref{thm:akh}.

\subs{Related work} Barannikov pairs were introduced in \cite{Bar} (see
\cite{EH} for a recent survey). A close idea of construction of Bruhat numbers
over $\Q$ appeared independently in \cite{Vit2}. In \cite{UZ} authors prove a
theorem analogous to Barannikov's in the setting of complexes over the Novikov
ring, which is useful in symplectic topology. Translation of Bruhat numbers to 
this setting is a current work in progress. Reidemeister torsion in the setting
of Morse-Novikov theory is studied in \cite{Hut1,Hut2,Hut3}.

\subs{Organization of the paper} The first two sections provide an algebraic
foundation for the further Morse-theoretical results. Namely, in
\Cref{sec:enh_spaces} we do the neccesary linear algebra and emphasize
connection with the Bruhat decomposition. In \Cref{sec:enh_comp} similar in
spirit constructions are presented in realm of homological algebra over a field.
\Cref{sec:morse,sec:br_tors,sec:one_par} contain (but not exhausted by) results
described in, respectively, \Cref{subs:morse,subs:br_tors,subs:one_par}.

The paper contains many constructions and we therefore use a special environment
for them. The text after the word ``construction'' describes input and output.
The actual algorithm is placed between signs $\vartriangleright$ and
$\vartriangleleft$. By default construction doesn't involve any choices, i.e.
output depends only on the input. If it is not the case, we indicate this
explicitly.

\subs{Acknowledgements} P. Pushkar is supported by Russian Foundation for Basic
Research under the Grants RFBR 18-01-00461 and supported in part by the Simons
Foundation. M. Tyomkin’s research was carried out within the framework of the
Basic Research Program at HSE University and funded by the Russian Academic
Excellence Project ``5-100''. We are grateful to Leonid Rybnikov for fruitful
discussions. We thank Andrei Ionov for bringing our attention to the fact that
the word ``enhanced'' has been used since \cite{BB} for distinguishing, in
particular, unitriangular group from upper triangular. Second author thanks
Sergey Melikhov for comments on the first version. He also thanks Petr Akhmetev
for numerous explanations of the paper \cite{Akhm}. He thanks Anton Ayzenberg
for the reference \cite{Cayley}.

\section{Enhanced vector spaces}
\label{sec:enh_spaces}
In this section we define and study enhanced vector spaces~--- a notion which we
rely on in \Cref{sec:enh_comp}. All the constructions lie within the scope of
linear algebra. Moreover, our main statement here (\Cref{lem:class_space}) may
be formulated exclusively in terms of matrices, which is done right below
(\Cref{lem:class_space_coord}). Later in \Cref{sec:enh_comp} we proceed
similarly in the setting of chain complexes over a field.

\subs{Formulation of results} In this subsection we introduce main definitions
of this section and formulate the main \Cref{lem:class_space}.

We start with the coordinate formulation. Let $n$ and $m$ be two natural numbers
fixed once and for all throughout this section. Fix also a base field $\F$, all
matrices are assumed to be over it.

Let $T_n$ be the group of unitriangular matrices, i.e. upper
triangular $n \times n$ matrices with ones on the diagonal. The group
$T_n \times T_m$ acts on the set $\text{Mat}_{n,m}$ of all $n \times m$
matrices: $X \mapsto AXB^{-1}$. Since one has commuting actions of both $T_n$
and $T_m$, the orbit space is usually denoted as
${T_n} \backslash \text{Mat}_{n,m} / {T_m}$. Note that two $n \times m$ matrices
lie in the same $T_n$ orbit if and only if one can be obtained from another by a
succesive performing of the following elementary operation: add to one row a
scalar multiple of another one, provided that the latter is higher than the
former. The same goes for $T_m$ and column operation.

\begin{defin}
  An $n \times m$ matrix is called a \emph{rook} matrix if in every row and in
  every column there is at most one non-zero entry.
\end{defin}

\begin{lem}
\label{lem:class_space_coord}
Every orbit in ${T_n} \backslash \text{Mat}_{n,m} / {T_m}$ contains exactly one
rook matrix. \qed
\end{lem}

The classical Bruhat decomposition is obtained from the above statement in two
steps:
\begin{enumerate*}[label=\arabic*)] \item restrict to the case $n=m$ and
  consider only non-degenerate square matrices, \item replace $T_n$ by an upper
  triangular group. \end{enumerate*} Keeping in mind the slightly greater level
of generality we propose the term ``enhanced Bruhat decomposition'' (see,
however, \cite{BB}). Note that rook matrix stores, in particular, the set of
elements from $\F^*$, in contrast to the matrix of permutation in the classical
case. The proof, however, goes along the same lines. See \cite{FultonHarris} for
an in-depth discussion. See also \cite{DPS} for a close in spirit generalization
of Bruhat decomposition. 

We will now introduce the main notion of the present section. Several basic
facts about it will be presented further in
\Cref{subs:constr_rook,subs:proof_spaces}. Often Bruhat decomposition is proven
using inductive arguments. We tried to refrain from those during the course of
this section (or, at least, hide them under the carpet of explicit
constructions).

\begin{defin}
\label{def:enh_space}
Let $V$ be a vector space over $\F$. An \emph{enhancement $\kappa$ on a vector
  space} $V$ is a choice of two structures:
\begin{enumerate}[label=\arabic*)]
\item a full flag on $V$, i.e. a sequence of subspaces
  $0=V^{0}\subset V^{1}\subset \ldots \subset V^{\dim V}=V$ s.t.
  $\dim(\sfrac{V^{s}}{V^{s-1}})=1$, $s \in \{ 1, \dots, \dim V \}$;
\item a non-zero element $\kappa_{s}$ in a one-dimensional vector space
  $\sfrac{V^{s}}{V^{s-1}}$, $s \in \{ 1, \dots, \dim V \}$.
\end{enumerate}
A vector space $V$ with an enhancement will be called an \emph{enhanced vector
  space} and denoted \mbox{as $(V,\kappa)$}.
\end{defin}

\begin{defin}
  Let $(V, \kappa)$ and $(W, \mu)$ be two enhanced vector spaces of the same
  dimension. We say that they are \emph{isomorphic} (and use the symbol $\simeq$
  for this) if there exist an isomorphism of vector spaces
  $\phi \colon V \xto{\sim} W$ s.t.
\begin{enumerate*}[label=\arabic*)]
\item $\phi(V^s) = W^s$;
\item $\tilde{\phi_s}(\kappa_s)=\mu_s$, where
\end{enumerate*} %
\[ \tilde{\phi_s} \colon \sfrac{V^s}{V^{s-1}} \to \sfrac{W^s}{W^{s-1}} \] is
a map of quotient vector spaces induced by $\phi$.
\end{defin}

By a basis of a finite-dimensional vector space $V$ we will mean a linearly
ordered set of generators (zero vector space has empty set as its only basis).
Given a basis $v=(v_1,\dots,v_{\dim V})$ of $V$ one constructs an enhanced
vector space $(V,\kappa(v))$ in the following straightforward way. \mbox{For
  $s \in \{ 1, \dots, \dim V \}$} set $V^s\colonequals \<\! v_1,\dots,v_s\!\>$
and $\kappa(v)_s\colonequals p_s(v_s)$, where
$p_s\colon V^s\to\sfrac{V^s}{V^{s-1}}$ is a standard projection. By a basis of
an enhanced vector space $(V,\kappa)$ we will mean a basis $v$ of $V$ s.t.
$(V,\kappa)\simeq (V,\kappa(v))$.

The next lemma is equivalent to \Cref{lem:class_space_coord}.
\begin{restatable}{lem}{LemClassSpaceRestate}
\label{lem:class_space}
Let $(V, \kappa)$ and $(W, \mu)$ be two enhanced vector spaces and
$A \colon V \to W$ be a linear map. There exists a basis $v$ (resp. $w$) of an
enhanced vector space $(V, \kappa)$ (resp. $(W, \mu)$) s.t. the matrix of $A$ in
these bases is a rook matrix. Moreover, this rook matrix is uniquely defined.
\end{restatable}

The change of basis in $(V, \kappa)$ (resp. $(W, \mu)$) results in
multiplication of matrix of $A$ by a matrix from $T_{\dim V}$ (resp.
$T_{\dim W}$). So, \Cref{lem:class_space} describes orbits in
${T_{\dim W}} \backslash \text{Mat}_{\dim W, \dim V} / {T_{\dim V}}$.

We will stick to the above formulation. It is possible to state the same without
appealing to any bases whatsoever; this is done in \Cref{subs:term_dig}. We call
non-zero elements of the mentioned rook matrix ``Bruhat numbers'' of a map
between two enhanced vector spaces.

\begin{rem}
  \label{rem:field_ext}
  If the field $\mathbb K$ is an extension of $\F$ then one may consider $A$ as
  a map between vector spaces over $\mathbb K$. It's plain to see that rook
  matrix won't change after this operation. Indeed, multiplication of matrices
  only involves additions and multiplications.

  Usually what we are given in topological setup is a matrix over $\Z$ (note
  that this is not the case in \Cref{subs:rt_recall,subs:rt_bruhat}). We then
  choose some field $\F$ and merely consider this matrix over this field. It
  follows from the previous paragraph that it is enough to consider only $\Q$
  and $\F_p$.
\end{rem}

\begin{rem}  
\label{rem:bases_not_unique}
In \Cref{lem:class_space} bases $v$ and $w$ themselves need not be
unique. %
\end{rem}

\subs{Construction of a rook matrix}
\label{subs:constr_rook}
In this subsection we associate a rook matrix to a given map between enhanced
vector spaces. In \Cref{subs:proof_spaces} we will show that this is the same
matrix as the one addressed in \Cref{lem:class_space}.

We will need the following construction as a preliminary step.

\begin{con}
\label{con:surj_enh}
Let $(V, \kappa)$ be an enhanced vector space and $A \colon V \onto W$ be a surjective
map of vector spaces. We will now construct an induced enhancement on $W$.

\triright For $s \in \{ 1, \dots, \dim V \}$ define $\phi_s$ to be the
composition 
\begin{center}
  \begin{tikzcd}
    V^s \arrow[r, hook] & V \arrow[r, two heads, "A"] & W. 
  \end{tikzcd}
\end{center}
Take any $s$ s.t. $\dim \im \phi_s = \dim \im \phi_{s-1} + 1$ and denote this
number by $t$. Set $W^t$ to be $\im \phi_s$. This defines a full flag on $W$,
i.e. a vector space $W^t$ for any $t \in \{ 1, \dots, \dim W \}$. We now need to
produce an element $\mu_t$ in the vector space $\sfrac{W^t}{W^{t-1}}$; this vector
space coincides with $\sfrac{\im \phi_s}{\im \phi_{s-1}}$. Define $\mu_t$ to be
$\tilde {\phi_s}(\kappa_s)$, where
\[
  \tilde {\phi_s} \colon \sfrac{V^s}{V^{s-1}} \xto{\sim} \sfrac{\im
    \phi_s}{\im \phi_{s-1}}
\]
is an isomorphism of quotient vector spaces induced by $\phi_s$. We have obtained an
enhanced vector space $(W, \mu)$. \trileft
\end{con}

The following proposition will be used later in \Cref{subs:proof_spaces}.

\begin{prop}
  \label{prop:surj_enh}
  Let $(V, \kappa)$ be an enhanced vector space and $A \colon V \onto W$ be a
  surjective map of vector spaces. \Cref{con:surj_enh} produces an enhanced vector space
  $(W, \mu)$. We claim that for any basis of $(W, \mu)$ one can find a basis of
  $(V, \kappa)$ s.t. $A$ maps each basis element to either zero or another basis
  element.
\end{prop}
\begin{proof}
  We continue using notations from \Cref{con:surj_enh}. Let
  $w=(w_1, \dots, w_{\dim W})$ be the given basis of $(W, \mu)$. Take any
  $s \in \{1, \dots, \dim V \}$. The difference
  $\dim \im \phi_s - \dim \im \phi_{s-1}$ is either zero or one. In the former
  case set $v_s$ to be any vector from $V^s \setminus V^{s-1}$ s.t. its class in
  $\sfrac{V^s} {V^{s-1}}$ coincides with $\kappa_s$. In the latter case set
  $v_s$ to be any preimage of $w_{\dim \im \phi_s}$ under $A$. It is
  straightforward to check that the basis $(v_1, \dots, v_{\dim V})$ satisfies
  the desired property.
\end{proof}
\Cref{con:inj_enh} and \Cref{prop:inj_enh} are analogous statements for the case
of injective map. In order to proceed to the construction of a rook matrix we
need two more definitions. For a non-zero element $v\in V$ of an enhanced vector
space $(V,\kappa)$ we define, first, a function $\Ht(v)$ (stands for height) to
be minimal $s$ s.t. $v \in V^s$. Second, we define a function $\cf(v)$ (stands
for coefficient) to be equal to
$\lambda \in \F^* \colonequals \F \setminus \{0 \}$ s.t.
$p(v)= \lambda \kappa_{\Ht(v)}$, where $p$ is a projection
$V^{\Ht(v)} \onto \sfrac{ V^{\Ht(v)} }{ V^{\Ht(v)-1} }$.

\begin{con}
\label{con:rook}
Given a map $A \colon V \to W$ between enhanced vector spaces $(V, \kappa)$ and
$(W, \mu)$ we will now construct a rook matrix $R$ of size
$(\dim W) \times (\dim V)$.

\triright Fix any $s \in \{ 1, \dots, \dim V \}$. Consider a surjective map
$S \colon W \onto \sfrac{W}{A(V^{s-1} ) }$ and use \Cref{con:surj_enh} to get an
enhanced vector space $(\sfrac{W} {A(V^{s-1} ) } , \tilde{\mu} )$. Consider now an
element $\tilde{A} (\kappa_s) \in \sfrac{W} {A(V^{s-1} ) }$, where
\[ \tilde{A} \colon \sfrac{V^s} {V^{s-1} } \to \sfrac{W} {A(V^{s-1} ) }\] is
a map of quotient vector spaces induced by the restriction $A|_{V^s}$. If
$\tilde{A} (\kappa_s) = 0$ then the $s^{\text{th} }$ column of $R$ is set to
be zero. Otherwise, let $\lambda$ and $t'$ be, respectively, coefficient and
height of $\tilde{A} (\kappa_s)$. Let $t \in \{ 1, \dots, \dim W \}$ be the
only number satisfying the condition $\dim S(W^t) = \dim S(W^{t-1} ) + 1 = t'$.
Finally, we set $R_{t,s}$ to be $\lambda$ and all the other entries in the
$s^{\text{th} }$ column of $R$ to be zero.

It is straightforward to check that in each row there is at most one non-zero
element, i.e. $R$ is indeed a rook matrix. \trileft
\end{con}

\begin{rem}
  Informally, the fact that the entry $R_{t,s}$ of a rook matrix $R$ is non-zero
  means that the image of $\kappa_s$ under $A$ first appears in the
  $t^{\text{th}}$ flag space in $W$.
\end{rem}

\subs{Terminological digression}
\label{subs:term_dig}
In this subsection we introduce a bit of terminology which will be useful for
understanding the content of \Cref{sec:enh_comp}.

Let $X$ and $Y$ be two sets and $\sim$ be an equivalence relation on $X$. If a
map $g \colon X \to Y$ is constant on the equivalence classes, then we say that
$g(x)$ is an invariant of some element $x \in X$. If, moreover, the induced map
$\tilde{g} \colon \sfrac{X} {\sim} \to Y$ is a bijection of sets then we say
that $g(x)$ is a full invariant.

Our next goal is to introduce a certain equivalence relation on the set of maps
between fixed enhanced vector spaces. By an automorphism of an enhanced vector
space $(V, \kappa)$ we will mean an isomorphism from $(V, \kappa)$ to itself. We
say that two maps $A$ and $B$ between enhanced vector spaces $(V, \kappa)$ and
$(W, \mu)$ are equivalent if there exist an automorphism $C_1$ (resp. $C_2$) of
$(V, \kappa)$ (resp. $(W, \mu)$) s.t. $C_2 A C_1 = B$.

\begin{rem}
\label{rem:equiv_bases}
  Note that $A$ and $B$ are equivalent if and only if there exist bases $v_a$
  and $v_b$ of $(V, \kappa)$ and bases $w_a$ and $w_b$ of $(W, \mu)$ s.t. the
  matrix of $A$ in bases $v_a$ and $w_a$ coincides with the matrix of $B$ in
  bases $v_b$ and $w_b$.
\end{rem}

Note that \Cref{con:rook} provides an invariant of a map between enhanced vector
spaces considered up to equivalence. This invariant takes values in the set of
rook matrices. It will follow from \Cref{subs:proof_spaces} that
\Cref{lem:class_space} can be reformulated as follows.
\begin{lem}
  Let $(V, \kappa)$ and $(W, \mu)$ be two fixed enhanced vector spaces and
  $A \colon V \to W$ be some linear map. Then the corresponding rook matrix
  provided by \Cref{con:rook}
  is a full invariant of a map considered up to equivalence. \qed
\end{lem}

\subs{Proof of the main \Cref{lem:class_space}}
\label{subs:proof_spaces}
In this subsection we prove \Cref{lem:class_space}. First, we prove the partial
case when the map in question is an isomorphism. Second, we derive the general
statement from it.

First of all, recall the lemma itself.

\LemClassSpaceRestate*  

See \Cref{subs:tors} for a context. 

We will now deduce uniqueness from the existence. Suppose the map
$A$ is represented by a rook matrix $R$ in some bases $v$ and $w$. Then one
checks straightforwardly that the \Cref{con:rook} produces the same matrix $R$
as an output. Therefore $R$ is an invariant of a map $A$ and we're done. The
rest of this subsection is devoted to proving the existence part.

The next proposition is a partial case which will be used later.

\begin{prop}
  \label{prop:bruhat_iso}
  Let $(V, \kappa)$ and $(W, \mu)$ be two enhanced vector spaces and
  $A \colon V \xto{\sim} W$ be an isomorphism. There exists a basis $v$ (resp.
  $w$) of an enhanced vector space $(V, \kappa)$ (resp. $(W, \mu)$) s.t. the
  matrix of $A$ in these bases is a rook matrix.
\end{prop}
\begin{proof}
  By a jump of a function $g \colon \{ 0, \dots, N \} \to \Z_{\geq 0}$, where
  $N \in \Z_{\geq 0}$ we will mean a number $x > 0$ s.t. $g(x) = g(x-1) + 1$.
  Fix any $s \in \{1, \ldots, \dim V \}$. Consider now a function
  $x \mapsto \dim (A(V^s) \cap W^x)$, for $x \in \{0, \ldots, \dim W \}$ (recall
  that $\dim V =\dim W$). It has exactly $s$ jumps. Moreover, every jump of a
  function $x \mapsto \dim (A(V^{s-1}) \cap W^x)$ is also a jump of the function
  under consideration. Therefore, the latter function has exactly one ``new''
  jump, call it $t$. It follows from the fact that $t$ is a jump that
  $A(V^s) \cap (W^t \setminus W^{t-1}) \ne \emptyset$; take any element $w$ from
  this set. It follows from the fact that $t$ is actually a new jump that
  $A^{-1}(w) \in V^s \setminus V^{s-1}$.

  By performing the above operation for all possible $s$ we construct a basis of
  $W$ and, by taking a preimage, a basis of $V$. The end of the preceeding
  paragraph implies that after appropiate reordering and rescaling these bases
  are bases of enhanced vector spaces $(V, \kappa)$ and $(W, \mu)$. The
  statement follows.
\end{proof}
\begin{rem}
  This is a known proof of the Bruhat decomposition for $GL_n$ adapted to our
  enhanced setting.
\end{rem}

\begin{con}
\label{con:inj_enh} 
Let $(W, \mu)$ be an enhanced vector space and $A \colon V \into W$ be an
injective map of vector spaces. We will now construct an induced enhancement on
$V$.

\triright For $s \in \{ 1, \dots, \dim W \}$ define $\phi_s$ to be the
composition
\begin{center}
  \begin{tikzcd}
    V \arrow[r, hook, "A"] & W \arrow[r, two heads] & \sfrac{W}{W^s}. 
  \end{tikzcd}
\end{center}
Take any $s$ s.t. $\dim \ker \phi_s = \dim \ker \phi_{s-1} + 1$ (call this
number $t$). Set $V^t$ to be $\ker \phi_s$. This defines a full flag on $V$,
i.e. a vector space $V^t$ for any $t \in \{1, \dots, \dim V \}$. We now need to
produce an element $\kappa_t$ in the vector space $\sfrac{V^t}{V^{t-1}}$, which
coincides with $\sfrac{\ker \phi_s}{\ker \phi_{s-1}}$. By identifying $V$ with
$\im A$, we say that $\ker \phi_s \subset W^s$. Define $\kappa_t$ to be
$\alpha^{-1}(\mu_s)$, where
\[
  \alpha \colon \sfrac{\ker \phi_s}{\ker \phi_{s-1}} \xto{\sim} \sfrac {W^s}
  {W^{s-1}}
\]
is an isomorphism of quotient vector spaces induced by the mentioned inclusion.
We have obtained an enhanced vector space $(V, \kappa)$. \trileft
\end{con}

\begin{prop}
\label{prop:inj_enh}
Let $(W, \mu)$ be an enhanced vector space and $A \colon V \into W$ be an
injective map of vector spaces. \Cref{con:inj_enh} produces an enhanced vector
space $(V, \kappa)$. We claim that for any basis of $(V, \kappa)$ one can find a
basis of $(W, \mu)$ s.t. $A$ maps each basis element to another basis element.
\end{prop}
\begin{proof}
  We continue using notations from \Cref{con:inj_enh}. Let
  $v=(v_1, \dots, v_{\dim V})$ be a given basis of $(V, \kappa)$. Take any
  $s \in \{ 1, \dots, \dim W \}$. The difference
  $\dim \ker \phi_s - \dim \ker \phi_{s-1}$ is either zero or one. In the former
  case set $w_s$ to be any vector from $W^s \setminus W^{s-1}$ s.t. its class in
  $\sfrac{W^s} {W^{s-1}}$ coincides with $\mu_s$. In the latter case set $w_s$
  to be $A(v_{\dim \ker \phi_s})$. It is straightforward to check that the basis
  $(w_1, \dots, w_{\dim W})$ satisfies the desired property.
\end{proof}

\begin{proof}[Proof of \Cref{lem:class_space}]
Uniqueness is shown in the beginning of the present subsection. To show the
existence, consider the composition of three maps:
\[
  V \onto \sfrac{V}{\ker A} \xto{\sim} \im A \into W.
\]
Induce enhancement on $\sfrac{V}{\ker A}$ from $V$ via \Cref{con:surj_enh} and
on $\im A$ from $W$ via \Cref{con:inj_enh}. Apply \Cref{prop:bruhat_iso} to the
middle map to obtain bases $\tilde{v}$ and $\tilde{w}$ of its source and
target respectively. Apply now \Cref{prop:surj_enh} to the basis $\tilde{v}$
to get a basis $v$ of $V$. Apply \Cref{prop:inj_enh} to the basis
$\tilde{w}$ to get a basis $w$ of $W$. By construction, bases $v$ and $w$
are the desired ones.
\end{proof}
\subs{On a matrix of a map between enhanced vector spaces}
\label{subs:matr_spaces}
In this subsection we give several properties of a matrix of a map between
enhanced vector spaces, written in appropriate basis.

\begin{con}
Let $R$ be a rook $n \times m$ matrix. We will now define a subset $\Ll(R)$ of a
set $\text{Mat}_{n,m}$ of all $n \times m$ matrices.
\end{con}
\triright Let $M \in \text{Mat}_{n,m}$ be a matrix. We say that its entry
$M_{i,j}$ is covered if there exists a pair of indices $(i',j')$ s.t. the
following two conditions hold:
\begin{enumerate}[label=\arabic*)]
\item $R_{i',j'} \neq 0$,
\item
  $(i < i' \; \text{AND} \; j \geq j') \; \text{OR} \; (i \leq i' \; \text{AND}
  \; j > j' )$.
\end{enumerate}
The matrix $M$ is said to be in $\Ll(R)$ if the the following two conditions
hold:
\begin{enumerate}[label=\arabic*)]
\item if the entry $M_{i,j}$ is not covered and $R_{i,j} = 0$ then it equals to
  zero,
\item if the entry $M_{i,j}$ is not covered and $R_{i,j} \neq 0$ then $M_{i,j} =
  R_{i,j}$. \trileft
\end{enumerate}
Here is an example, for $\F = \Q$, of the matrix $R$ and the general form of a
matrix $M$ from the \mbox{set $\Ll(R)$}:
\[
 R = 
\begin{pmatrix}
    0 & 0 & 4 \\
    3 & 0 & 0 \\
    0 & 0 & 0 \\
    0 & 2 & 0 \\
\end{pmatrix}
, M =
\begin{pmatrix}
    * & * & * \\
    3 & * & * \\
    0 & * & * \\
    0 & 2 & * \\
\end{pmatrix}.
\]
In the case of classical Bruhat decomposition analogous set is nothing but a
Bruhat cell. The next proposition is straightforward and well-known in the
classical case.
\begin{prop}
\label{prop:bruhat_cell}
  Let $(V, \kappa)$ and $(W, \mu)$ be two enhanced vector spaces and
  $A \colon V \to W$ be some linear map. Let also $v$ (resp. $w$) be some basis
  of enhanced vector space $(V, \kappa)$ (resp. $(W, \mu)$). Then the matrix of
  $A$ in the bases $v$ and $w$ belongs to $\Ll(R)$, where $R$ is a rook matrix
  from \Cref{lem:class_space}.
\end{prop}

The next statement links properties of integral matrix (considered up to
unitriangular change of basis) and its enhanced Bruhat decomposition over $\Q$.
Let $\text{Mat}_{n,m}(\Z)$ be the set of all $n \times m$ matrices over $\Z$. In
what follows we will sometimes view it as a subset of matrices over $\F = \Q$
without mentioning this explicitly. Let also $T_n(\Z)$ be the group of
unitriangular matrices \mbox{over $\Z$}. The group $T_n(\Z) \times T_m(\Z)$ acts
on the set $\text{Mat}_{n,m}(\Z)$. The next proposition follows from
\Cref{prop:bruhat_cell}.

\begin{prop}
  \label{prop:matr_spaces}
  Consider $M \in \text{Mat}_{n,m}(\Z)$. Then any element $M'$ from the orbit
  ${T_n(\Z) \cdot M \cdot T_m(\Z)}$ lies in the set $\Ll(R)$, where $R$ is a
  rook matrix over $\F = \Q$ associated \mbox{with $M$}.
\end{prop}
As a corollary one gets that at least one non-zero entry of $R$ is integer. It
also follows that there is at least one pair of indices $(i,j)$ s.t. the entry
$M'_{i,j}$ is the same for any $M'$ from the mentioned orbit.

\subs{Geometric approach to enhancements, dual enhancement and enhancements on a
  subspaces and quotient vector space}

We recall that for an affine subspace $A$ in a vector space $V$ a set
$\{a-b|a,b\in A\}$ is a vector subspace of $V$ associated with $A$, it is called
the direction of $A$. We will denote this vector subspace by $\mathrm{dir} A$.

\begin{defin} An enhancement of a vector space $V$ is a subset of $V$, which we
  will also call $\kappa$, such that: $\kappa$ is a disjoint union of affine
  subspaces $A^i$ ($A^i\cap A^j=\emptyset$ for $i \ne j$) of dimension $i-1$
  for $i\in\{1,\dots,\dim V\}$ and such that $0 \notin \kappa$ and
  $\mathrm{dir} A^{i+1} = \mathrm{span}(A^i)$ for each $i \in \{ 1, \dots, \dim V \}$,
  where $A^{\dim V+1}$ equals to $V$.
\end{defin}

The collection $A^1, \dots, A^{\dim V}$ of affine subspaces is uniquely recovered
from the enhancement $\kappa$. For the corresponding algebraically defined
enhancement $\kappa$ flag spaces are
$V^i = \mathrm{span} (A^i) = \mathrm{dir} A^{i+1}$ and non-zero elements $\kappa^s$
are images of $A^s$ under natural quotient maps $V^s\to V^s/V^{s-1}$. The full
flag $V^1 \subset \ldots \subset V^{\dim V} = V$ of directions of enhancement we
will denote by $\mathrm{dir} \kappa$.

Dual vector space $V^*$ to an enhanced vector space $(V,\kappa)$ is naturally
enhanced as well: a dual enhancement $\kappa^*$ on $V^*$ is, by definition, a
union of the following spaces
\[
\Uni (A^i)=\{p\in V^*| p(A^i)=1 \}.
\]
Note that $\dim \Uni ({A^i}) =\dim V-i-1$.  
Note also that dual flag to 
\[
0= \mathrm{dir} A^1\subset \ldots \subset \mathrm{dir} A^{\dim V}\subset V,
\]
which by definition consists of annihilator spaces to flag spaces
\[
0=\Ann(V)\subset \Ann (\mathrm{dir} A^{\dim V-1})\subset \ldots 
\subset \Ann (\mathrm{dir} A^1) \subset V^*
\] 
coincides with
\[
  0= \mathrm{dir} \Uni (A^{\dim V})\subset \mathrm{dir} \Uni (A^{\dim
    V-1})\subset\ldots \subset \mathrm{dir} \Uni (A^1)\subset V^*.
\]

If $\kappa=\kappa(v)$,  where $v=(v_1,\ldots,v_n)$ is ordered basis of
V,  then $\kappa^*=\kappa(g_n,\ldots,g_1)$ for a dual basis
$\{g_1,\ldots,g_n\}$ of $V^*$ (such that $g_i(v_j)=\delta_{ij}$).

If $L\subset V$ is a linear subspace of an enhanced vector space $(V,\kappa)$ then one
can easily show that the set $\kappa_L=\kappa \cap L$ is an enhancement on the
vector space $L$. Hence, for an injective linear map $i\colon N \to V$ the
preimage $i^{-1}(\kappa)$ is also an enhancement, and we denote it by
$i^*\kappa$.

A quotient vector space $\sfrac{V}{L}$ of an enhanced vector space $(V,\kappa)$ by a subspace
$L$ is also enhanced vector space with an enhancement $\kappa_{/L}$ defined by the
following construction.For a natural quotient map
$\psi \colon V \to \sfrac{V}{L}$ the dual map
$\psi^*\colon (\sfrac{V}{L})^* \to V^*$ is an injection. The vector space $V^*$ is
enhanced with the dual enhancement $\kappa^*$, hence one can consider the
enhancement $(\psi^*)^*\kappa$ on $\sfrac{V}{L}^*$. The dual enhancement
$((\psi^*)^*\kappa)^*$ is an enhancement on the vector space $(\sfrac{V}{L})^{**}$
which we will identify with $\sfrac{V}{L}$ and the corresponding enhancement
denote by $\kappa_{/L}$ and we call it as enhancement $\kappa$ quotient by~$L$.
 
\subs{Construction of enhancement from flag and enhancement}
\label{subs:geom_enh}
We say that two
enhancements $\kappa_1$ and $\kappa_2$ (on the same vector space) are parallel
iff flags $\mathrm{dir} \kappa_1$ and $ \mathrm{dir} \kappa_2$ coincide. Note
that a set of all enhancements parallel to a given one is naturally a torus
$(\F\setminus \{ 0\})^{\dim V}$. Namely, for every pair $(\kappa_1,\kappa_2)$ of
parallel enhancements we construct an ordered collection of non-zero numbers
$\Lambda(\kappa_1,\kappa_2)=(\lambda_1(\kappa_1,\kappa_2),\ldots,\lambda_{\dim
  V}(\kappa_1,\kappa_2))$, where each number
$\lambda_i=\lambda_i(\kappa_1,\kappa_2)$ is given by the relation
$\kappa^i_1=\lambda_i \kappa^i_2$.
 
The main construction of this section is the following. We will correspond to
pair $(f,\kappa)$ consisting of a full flag
$f=W^1\subset\ldots\subset W^{\dim V}$ on $V$ and an enhancement $\kappa$ on $V$
an enhancement $\chi = \chi(f,\kappa)$ on $V$ such that $\mathrm{dir}\chi=f$. It
is given by the following inductive procedure, we will construct affine spaces
$B^1\subset W^1$, $B^2\subset W^2$,... step by step. Denote by
$A^1,...,A^{\dim L}$ a collection of affine subspaces, generating $\kappa$. Let
$\tilde A^i = {\Uni (A^{\dim V -i+1})}$. Let us construct a point $B^1$ --
geometrically it is the intersection of $\kappa$ and $W^1$, but we define it in
the following way: consider the smallest $k_1$ such that
$\tilde A^{k_1}|_{W_1}\ne 0$ and let
\[
B^1=\{x\in W^1| p(x)=1 \forall p\in \tilde  A^{k_1}|_{W^1}\}.
\]
To construct $B^2$ we take smallest $k_2$ such that
$\tilde A^{k_1}\wedge \tilde A^{k_2}|_{W^2}\ne \{0\}$ (exterior product
of sets is a set of all exterior products of elements from initial sets)
\[
  B^2=\{x\in W^2| p(x)=1 \forall p\in \tilde A^{k_1}|_{W^2}\cup \tilde
  A^{k_2}|_{W^2} : p|_{B^1}=0\}.
\]
To construct $B^{i+1}$ we take smallest $k_{i+1}$ such that
$\tilde A^{k_1}\wedge \tilde A^{k_2}\wedge\ldots \wedge\tilde
A^{k_{i+1}}|_{W^{i+1}}\ne 0$ and
\[
  B^{i+1}=\{x\in W^{i+1}| p(x)=1 \forall p\in \tilde A^{k_1}|_{W^{i+1}}\cup
  \tilde A^{k_2}|_{W^{i+1}}\cup\ldots\cup \tilde A^{i+1}|_{W^{i+1}} :
  p|_{B^1\cup B^2\cup\ldots \cup B^{i}}=0\}.
\]
Each $B^i$ is an affine space of the dimension $I-1$ and the set
$B^1\cup\ldots \cup B^{\dim V}$ is the desired enhancement $\chi(f,\kappa)$.
Also we get a permutation $\sigma(f,\kappa)= (k_1,\ldots,k_{\dim V})$.

\begin{rem} If $\mathrm{dir}\kappa = f$ then $\chi(f,\kappa)=\kappa$. If the
  field $\F$ is $\R$ or ${\mathbb C}$ and $\dim V> 1$ then the map
\[
(f,\kappa)\mapsto \chi (f,\kappa)
\]
is discontinuous, even functions
$\lambda_i(\kappa_1,\chi (\mathrm{dir}\kappa_1,\kappa_2))$ are discontinuous,
but their product
\[
  \lambda_1(\kappa_1,\chi (\mathrm{dir}\kappa_1,\kappa_2))\cdot\ldots\cdot
  \lambda_{\dim V}(\kappa_1,\chi (\mathrm{dir}\kappa_1,\kappa_2))
\]
is a continuous function -- its value is the determinant of any isomorphism $I$
such that $I(\kappa_1^s)=\kappa_2^s$ for any $s$.
\end{rem}

\subs{Enhancements and grassmanians and determinant} Suppose $V$ is a
vector space with two enhancements $\kappa, \chi$ then for every subspace $L
\subset V$ we get enhancements$\kappa_L, \chi_L$ on $L$ and we can
construct numbers $\lambda_1(\kappa_L, \chi_L),\ldots,\lambda_{\dim
L}(\kappa_L, \chi_L)$ and permutation
$\sigma(\mathrm{dir}\kappa_L,\chi_L)$  Consider  product

\[
d(\kappa_L,\chi_L)=(-1)^{\sigma(\mathrm{dir}\kappa_L,\chi_L)}\lambda_1(\kappa_L,
\chi_L)\ldots \lambda_{\dim L}(\kappa_L, \chi_L).
\]

By construction this function has only non-zero values. This function
has meaning of ordinary determinant in the following sense.  Consider
vector space $V$ with ordered basis $v$ and vector space $U$ with ordered basis $u$.
Then vector space $V\oplus U$ is endowed by ordered basis $(v,u)$ and by
ordered basis $(u,v)$ and hence one can construct enhancements
$\kappa(v,u)$ and $\kappa(u,v)$.  Then for any isomorphism $A\colon V
\to U$ its determinant equals  to $d_{\kappa(v,u)_L,\kappa(u,v)_L}$
for $L\subset V\oplus U$ being a graph of $A$.  If the map $A\colon V
\to U$ is not an isomorphism we still get a non-zero number and for 0
operator this number is $1$.   

\section{Enhanced complexes}
\label{sec:enh_comp}
In this section we define and study enhanced complexes~--- an algebraic object
which will carry a certain information about a strong Morse function (see
\Cref{subs:enh_comp_morse}). All the constructions lie within the scope of
homological algebra of chain complexes over a field. They are similar in spirit
to those in \Cref{sec:enh_spaces}. The purpose is that enhanced complex is a
useful algebraic container that stores some information about a strong Morse
function, see \Cref{sec:morse}.

\subs{Definition of an enhanced complex}
\label{subs:def_enh_comp}
In this subsection we define the object of study of this section.

\begin{defin}
\label{def:enh_complex}
Let $\C$ be a (chain) complex of vector spaces,
\[
 C_{n+1} = 0 \to C_n \xto{\d_n} \dots \xto{\d_1} C_0 \to 0 = C_{-1}.
\]
An \emph{enhancement $\kappa$ on a complex} $\C$ is an enhancement
$(\C_\bullet,\kappa)$ on a vector space
$\C_\bullet \colonequals \oplus_{j=0}^{n} C_j$ satisfying the condition that
each $\C^s_\bullet$ is a subcomplex of $\C$ (we will therefore write $\C^s$
instead of $\C^s_\bullet$ in order to stress the structure of a complex). A
complex with an enhancement will be called an \emph{enhanced complex} and
denoted as $(\C,\kappa)$.
\end{defin}
We call the number $n$ the dimension of $\C$ and denote it by $\dim \C$.

\begin{rem}
\label{rem:enh_c_k}
  \begin{enumerate}[label=\arabic*.]
  \item Recall that the aforementioned condition amounts to the following two:
      \begin{enumerate*}[label=\arabic*)]
      \item $\C^s$ is decomposed into the direct sum of graded components
        $\oplus_k C^s_k$ s.t. $ C^s_k \subset C_k$ for
        $k \in \{ 0, \dots, n \}$;
      \item $\d_k ( C^s_k ) \subset C^s_{k-1}$ for $k \in \{ 1, \dots, n \}$.
      \end{enumerate*}
  \item By \Cref{con:inj_enh} the vector space $C_k$ is also enhanced.
  \end{enumerate}
\end{rem}

\begin{rem}
  \label{rem:grading}
  For an enhanced complex $(\C,\kappa)$ the set $\{1,\dots,\dim \C_\bullet\}$ is
  $\Z_{\geq 0}$-graded: the degree $\deg s$ of $s$ is given by the only degree
  in which the complex $\sfrac{\C^s}{\C^{s-1}}$ is non-zero.
\end{rem}

\begin{defin}
  Let $(\C,\kappa)$ and $(\mathcal D,\mu)$ be two enhanced complexes with
  $\dim \C_\bullet = \dim {\mathcal D}_\bullet$. We say that they are
  \emph{isomorphic} (and use the symbol $\simeq$ for this) if there exist an
  isomorphism of complexes $\phi \colon \C \xto{\sim} \mathcal D$ s.t. the
  induced map $\C_\bullet \to {\mathcal D}_\bullet$ is the isomorphism on
  enhanced vector spaces.

\end{defin}

\begin{rem}
  As we will show in \Cref{sec:morse}, given a strong Morse function and
  suitable orientations one can construct an enhanced complex, which is
  well-defined up to isomorhism.
\end{rem}

\subs{Enhancement on $\H_\bullet(\C)$}
\label{subs:enhancement_on_h}  
In this subsection we construct enhancement on a homology of a certain class of
complexes, which includes enhanced ones.

For any complex $\C$ denote by $\H_\bullet(\C)$ the direct sum
$\oplus_j \H_j(\C)$. Let $(\C, \kappa)$ be an enhanced complex. Then for any $s$
the vector space $\H_\bullet(\C^s, \C^{s-1})$ is one-dimensional with a
preferred generator of degree $\deg s$ given by a class of relative chain
$\kappa_s \in \sfrac{\C^s} {\C^{s-1}}$.
\begin{con}
\label{con:enhancement_on_h}
Let $\C$ be a filtered (possibly infinite-dimensional) complex over a field
$\F$, $0 = \C^0 \subset \ldots \subset \C^N = \C$. Suppose that for any $s$
vector space $\H_\bullet(\C^s, \C^{s-1})$ is one-dimensional with a chosen
generator $h_s$. We will now construct an enhancement on a homology vector space
$\H_\bullet(\C)$.

\triright First, we will construct a filtration on $\H_\bullet(\C)$. For
$s \in \{ 0, \dots, N\}$ let $\iota_s \colon\H_\bullet(\C^s) \to \H_\bullet(\C)$
be a map induced by inclusion. Define the subset $H$ (which stands for homology)
of the set $\{1, \dots, N \}$ to be the set of all $s$ s.t.
$\dim \im \iota_s = \dim \im \iota_{s-1}+1$. Let $s_i$ be the $i^{\text{th}}$
element of $H$ (counting from 1) and $s_0$ be zero. The sequence of subspaces
$0 = \im \iota_{s_0} \subset \im \iota_{s_1} \subset \ldots \subset \im
\iota_{s_{\dim \H_\bullet(\C)}} = \H_\bullet(\C)$ is a full flag on
$\H_\bullet(\C)$ (follows from considering the exact sequence of a pair
$(\C^s, \C^{s-1})$ for all $s$). To complete the construction of enhancement we
will now produce an element from $\sfrac{\im \iota_{s_i}}{\im \iota_{s_{i-1}}}$
for a given $i \in \{ 1, \dots, \dim \H_\bullet( \C ) \}$. We denote $s_i$ by
$s$ for convenience. Let $\deg s$ denote the degree in which the graded vector space
$\H_\bullet(\C^s, \C^{s-1})$ is non-zero (this notation is coherent with the
case when $\C$ is an enhanced complex).

Consider the following diagram, with horizontal line being a portion of a long
exact sequence of a pair $(\C^s,\C^{s-1})$:
\[
  \begin{tikzcd}
    \H_{\deg s}    (\C^{s-1}) \arrow[r, "\theta"] \arrow[dr, "\iota_{s-1}"] &
    \H_{\deg s}    (\C^s)     \arrow[r, two heads, "p_*"] \arrow[d, "\iota_s"] &
    \H_{\deg s}    (\C^s, \C^{s-1}) \\ 
    & \H_\bullet   (\C) & h_s \arrow[u, symbol=\in]
  \end{tikzcd}
\]
We write $\iota_s$ both for a map $\H_\bullet(\C^s) \to \H_\bullet(\C)$ and for
its restriction to $\H_{\deg s}(\C^s)$. It follows from the definition of $H$
that $\dim \coker \theta = 1$, therefore $\ker p_*$ is a proper subspace of
$\H_{\deg s}(\C^s)$, which in turn implies that $p_*$ is surjective. Denote by
$p_*^{-1}(h_s)$ any preimage of $h_s$; it is defined up to elements from
$\ker p_*\simeq \im \theta$. Finally, the desired element is a class of
$\iota_s (p_*^{-1}(h_s))$ in the quotient space
$\sfrac{\im \iota_s}{\im \iota_{s-1}} \simeq \sfrac{\im \iota_s}{\im
  \iota_{s_{i-1}}}$ (mind that $\im \iota_{s_i-1} = \im \iota_{s_{i-1}}$); it is
well-defined. \trileft
\end{con}

\begin{rem}
  \label{rem:enh_on_h}
  Note that to a flag space of $\H_\bullet(\C)$ of dimension $d$ one can
  associate a number $s \in \{1, \dots, N \}$ as the unique solution of equation
  $\dim \im \iota_s = \dim \im \iota_{s-1} + 1 = d$. In other words, this is the
  smallest $s$ such that given flag space is contained in the image of
  $\iota_s$.

  If one is given a non-zero vector $v$ in the enhanced vector space
  $\H_\bullet(\C)$ then they may consider the flag space of least possible
  dimension containing $v$ (its dimension is $\Ht(v)$). Combining this with the
  previous paragraph one can associate a number $s$ with a vector $v$. We will
  make use of this association in \Cref{subs:bd}.
\end{rem}

For a detailed treatment of the mentioned long exact sequence see \cite{Lau}.
Without taking $\kappa_s$ into account it was first considered in \cite{Vit1}.
This preferred generator appeared independently in \cite{Vit2}. We denote
obtained enhancement as $(\H_\bullet(\C),\kappa_\H)$. Specializing the above
discussion to a fixed degree $k$ (and thus having $\deg s_i=k$) we get an
enhanced vector space $(\H_k(\C),\kappa_{\H_k})$; one may check that this is the
same enhancement as the one induced by inclusion $\H_k(\C) \into \H_\bullet(\C)$
via \Cref{con:inj_enh}. Note that the above procedure also gives enhancement on
$\H_\bullet(\C^s)$ for any $s$ (it will be crucial in what follows).

\begin{rem}
  \label{rem:fin_dim}
  The reason for the chosen level of generality is that one may take the input
  to be a complex of singular chains on a manifold equipped with a Morse
  function. The filtration is then given by the sublevel sets. See
  \Cref{subs:enh_comp_morse}, where B-data (described below in
  \Cref{subs:bd}) is extracted from the function this way.
\end{rem}

\begin{rem}
  In the case when $\C$ is an enhanced complex one may check that the following
  alternative construction produces the same enhancement on $\H_\bullet(\C)$.
  Consider the map $\d:\C_\bullet \to \C_\bullet$. Induce enhancements on
  $\ker \d$ and $\im \d$ via \Cref{con:inj_enh}. Induce enhacement on
  $\sfrac{\ker \d}{\im \d} = \H_\bullet (\C)$ via \Cref{con:surj_enh}.
\end{rem}

\subs{B-data}
\label{subs:bd}
In this subsection we introduce a certain data extracted from an enhanced
complex. This data is invariant under isomorphisms. In
\Cref{subs:classification} we show that this data is in fact a full invariant of
an enhanced complex considered up to isomorphism. The letter B stands
simultaneously for Barannikov, Bruhat and barcode.
\begin{wrapfigure}{R}{2cm}
  \includegraphics[width=2cm]{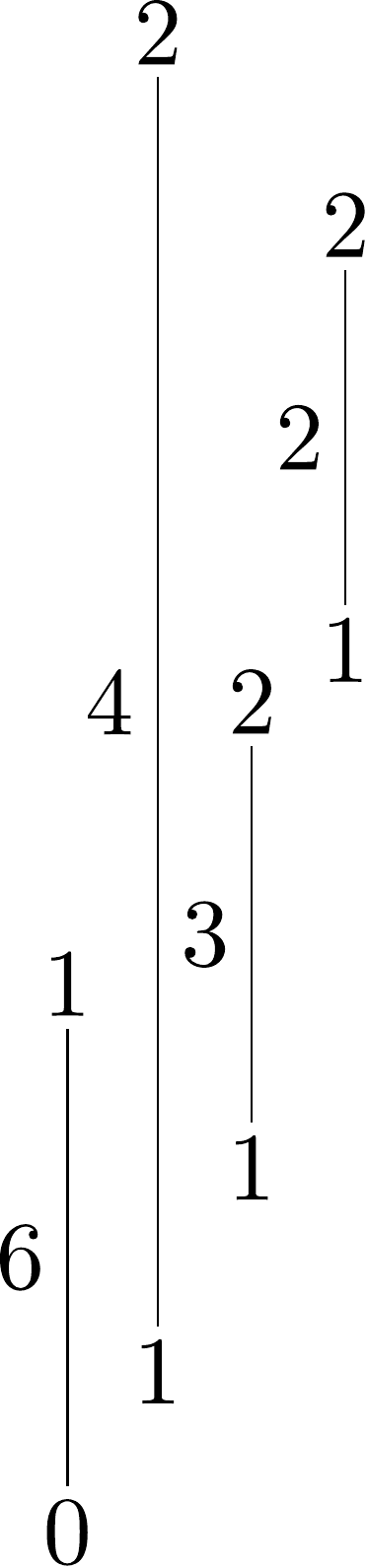}
  \caption{}
  \label{fig:b_data_example}
\end{wrapfigure}

\begin{rem}
  \label{rem:bd_gen}
  The extraction of B-data is done in the same way for the (more
  general) case of a complex $\C$ which satisfies conditions of
  \Cref{con:enhancement_on_h} (compare \Cref{rem:fin_dim}). Indeed, one has to
  merely replace the symbol $[\kappa_s]$ with $h_s$. In this case, however, we
  don't claim that this data is a full invariant; we don't even define any
  equivalence relation on the set of such complexes.
\end{rem}

We will now describe the B-data. It consists of several parts:
\begin{enumerate}[label=\roman*)]
\item \label{dataFirst} A non-negative integer $N$ along with a
  $\Z_{\geq 0}$-grading on a set $\{1, \dots, N \}$, denoted by $\deg$;
\item Decomposition of $\{ 1, \dots, N \}$ into the union of three disjoint sets
  $U, L, H$ (these letters stand for upper, lower and homological, for the 
  reasons described below);
\item Bijection $b \colon U \xto{1-1} L$ of degree $-1$ w.r.t. the grading. Map $b$
  must satisfy $b(s)<s$; 
\item \label{dataLast} A function $\lambda \colon U \to \F^*$. We write
  $\lambda_s$ for its value on $s \in U$.
\end{enumerate}
We call the image of $\lambda$ ``Bruhat numbers'' of an enhanced complex (see
the proof of \Cref{thm:classification} for the explanation). Two numbers $s$ and
$b(s)$ are said to form a Barannikov pair (or simply a pair). It's convenient to
think of each Bruhat number as being ``written'' on a Barannikov pair. Roughly
speaking, B-data is a decomposition of some subset of $\{ 1, \dots, N \}$ into
Barannikov pairs (the rest of the elements are homological). Each pair consists
of an upper element, a lower one and carries a Bruhat number. In other words,
B-data is a finite sequence of rook matrices $\{ R_k \}$ over $\F$, where $R_k$
is of size $( \# \{s | \deg s = k - 1 \} ) \times ( \# \{ s | \deg s = k \} )$
and $R_{k-1}R_k = 0$.

\Cref{fig:b_data_example} gives an example of B-data over $\Q$ and describes
pictorial format which we will use in future. Elements of the set
$\{1, \dots, N\}$ are drawn from bottom to top, pairs correspond to segments.
Either to the left or to the right of a segment we write a Bruhat number. The
degree of an element is written either above or below this element, whatever is
more convenient. In the example $N=8$, degree of $1$ is $0$, degree of $2,3,4$
and 6 is $1$ and degree of $5,7,8$ is 2. Next, $U = \{ 4,5,7,8 \}$,
$L = \{ 1,2,3, 6 \}$, $H = \emptyset$. Bruhat numbers are $6, 3, 2, 4$ (i.e.
values of $\lambda$ on 4,5,7,8 respectively). The map $b$ is defined by the
segments. Finally, two rook matrices are
\[
  R_1 =
  \begin{pmatrix}
    0 & 0 & 6 & 0 \\
  \end{pmatrix},
 R_2 = 
\begin{pmatrix}
    0 & 0 & 4 \\
    3 & 0 & 0 \\
    0 & 0 & 0 \\
    0 & 2 & 0 \\
  \end{pmatrix}.
\]

We will now show how to extract such a data from an enhanced complex
$(\C, \kappa)$. In future we will refer to it as a B-data of
$(\C, \kappa)$ and use the same letters $N, U, L, H, b$ and $\lambda$ for its
ingredients when necessary. We continue using notations introduced in
\Cref{subs:def_enh_comp,subs:enhancement_on_h}. Let
$\delta \colon \H_{\deg s}(\C^s, \C^{s-1}) \to \H_{\deg s - 1}(\C^{s-1})$ be the
connecting homomorphism. To begin with, set $N$ to be the number of filtration
components (counting from zero) and set the grading on $\{1, \dots, N\}$ to be
the one defined in \Cref{rem:grading}. Now for each $s \in \{ 1, \dots, N \}$
s.t. $\delta([\kappa_s])\neq 0$ do the following.
\begin{enumerate}[label=\arabic*)]
\item Put $s$ in $U$.
\item For any $t \in \{ 1, \dots, s-1 \}$ let
  $\iota_t^{s-1} \colon \H_{\deg s}(\C^t) \to \H_{\deg s}(\C^{s-1})$ be the map 
  induced by inclusion. Now choose the only $t$ s.t.
  $\dim \im \iota_t^{s-1} = \dim \im \iota_{t-1}^{s-1} + 1 =
  \Ht(\delta([\kappa_s]))$ (the function $\Ht(\cdot)$ here is taken w.r.t. the
  enhancement on $\H_{\deg s - 1 }(\C^{s-1})$ constructed from $\C^{s-1}$). See
  \Cref{rem:enh_on_h} for an informal meaning of such a $t$. Put this $t$ in
  $L$.
\item Set $\lambda(s)\colonequals \cf(\delta([\kappa_s]))$ w.r.t. the same
  enhancement.
\item Set $b(s)\colonequals t$.
\end{enumerate}
Using diagram chasing similar to that in \Cref{con:enhancement_on_h} one
verifies that such an operation is well-defined in a sense that, first, each
number will be put somewhere at most once and, second, that $b$ enjoys desired
properties, see \cite{Lau}. Those numbers in $\{ 1, \dots, N \}$ which were not
put anywhere by this operation are put in $H$. The extraction of B-data
\ref{dataFirst}-\ref{dataLast} is over, it's plain to see that it's invariant
under isomorphism.
\begin{rem}
\label{rem:bd_class_space}
  Here is another way to extract a B-data out of an enhanced complex. Apply
  \Cref{lem:class_space} to a differential $\d_k \colon C_k \to C_{k-1}$ viewed
  as a map between enhanced vector spaces (see \Cref{rem:enh_c_k}). This would
  give a sequence $\{ R_k \}$ of rook matrices and thus we're done. Although
  this way is shorter, we find it less instructive.
  
  Yet another approach is to use spectral sequence of a filtered (actually,
  enhanced) complex.
\end{rem}

Again, arguing as in \Cref{con:enhancement_on_h} one can show that
$\# \{ s \in H | \deg s = k \} = \dim \H_k(\C)$. We stress that everything
except $\lambda$ was essentially constructed in \cite{Bar}, while homological
language was first used in \cite{Vit1}. A close idea of construction of Bruhat
numbers over $\Q$ appeared independently in \cite{Vit2}. 

\begin{rem}
Consider once again the portion of a long exact sequence of a pair $(\C^s,
\C^{s-1})$:

\begin{center}
  \begin{tikzcd}
    \H_{\deg s}    (\C^s)     \arrow[r, "p_*"] &
    \H_{\deg s}    (\C^s, \C^{s-1}) \arrow[r, "\delta"] & 
    \H_{\deg s - 1}(\C^{s-1}). \\
  \end{tikzcd}
\end{center}
Since the middle term is one-dimensional, there are two cases possible:
\begin{enumerate}[label=\arabic*)]
\item The map $p_*$ is surjective whilst $\delta$ is zero. This case was used in
  \Cref{con:enhancement_on_h}.
\item The map $\delta$ is injective whilst $p_*$ is zero. This case was used in
  the extraction of B-data.
\end{enumerate}
\end{rem}

\begin{rem}
  \label{rem:4term}
  Let $s$ and $t$ be some elements from $\{ 1, \dots, N \}$. They form a
  Barannikov pair (i.e. $b(s) = t$) if and only if the following equations hold:
  \[
  \dim \H_\bullet(\C^{s-1}, \C^t) =
  \dim \H_\bullet(\C^s, \C^{t-1}) =
  \dim \H_\bullet(\C^{s-1}, \C^{t-1}) - 1 =
  \dim \H_\bullet(\C^s, \C^t) -1.
  \]
  This can be proven either by diagram chasing similar to that in
  \Cref{con:enhancement_on_h} or by \Cref{thm:classification} (see \cite{PC} and
  also \cite{Vit1}).
\end{rem}

We will now take coordinate viewpoint, which will be useful in formulation of
the classificational \Cref{thm:classification} in \Cref{subs:classification}. By
a basis of a chain complex $\C$ we will mean a basis
$(c_1,\dots,c_{\dim\C_\bullet})$ of a vector space $\C_\bullet$ s.t. each $c_s$
belongs to some $C_k$, where $k$ depends on $s$. By a basis of an enhanced
complex $(\C,\kappa)$ we will mean a basis $c$ of a chain complex $(\C,\d)$ s.t.
$(\C_\bullet,\kappa)\simeq (\C_\bullet,\kappa(c))$ (see \Cref{def:enh_complex}).
Every enhanced complex can be equipped with a basis.

Vice versa, given a basis $c$ of a chain complex $\C$ s.t. the span
$\langle c_1, \dots, c_s \rangle$ is a subcomplex (for each
$s \in \{1, \dots, \dim \C_\bullet \}$) one may construct an enhanced complex
$(\C, \kappa(c))$ by declaring
$(\C_\bullet, \kappa) \colonequals (\C_\bullet, \kappa(c))$.

\begin{rem}
  Note that two enhanced complexes $(\C, \kappa)$ and $(\mathcal{D}, \mu)$ are
  isomorphic if and only if the following holds:
\begin{enumerate}[label=\arabic*)]
\item $\dim \C_\bullet = \dim \mathcal{D}_\bullet$ and the gradings on
  $\{1, \dots, \dim \C_\bullet \}$ constructed from $(\C, \kappa)$ and
  $(\mathcal{D}, \mu)$ coincide (see \Cref{rem:grading}),
\item there exist bases $c$ of $(\C, \kappa)$ and $d$ of $(\mathcal{D}, \mu)$
  s.t. corresponding two matrices of differentials coincide (compare
  \Cref{rem:equiv_bases}).
\end{enumerate}
\end{rem}

\begin{defin}
  Let $(\C,\kappa)$ be an enhanced complex and $c$ be its basis. We call $c$ a
  \emph{Barannikov basis} if the matrix of differential
  $\d_k \colon C_k \to C_{k-1}$ in the basis $c$ is a rook matrix (for any $k$).
\end{defin}
  
The set of rook matrices $\{ R_k \}$ of differenials $\d_k$ doesn't depend on a
particular choice of a Barannikov basis. Indeed, this set is precisely the
B-data, which is an invariant of an enhanced complex. The purpose of
\Cref{subs:classification} is to show that every enhanced complex admits a
Barannikov basis. This will imply that B-data is a full invariant of an enhanced
complex considered up to isomorphism.

\subs{Classification of enhanced complexes}
\label{subs:classification}
In this subsection we prove classificational
\begin{thm}
  \label{thm:classification}
  For every enhanced complex $(\C,\kappa)$ there exists a Barannikov basis $c$.
  Moreover, the matrix of $\d$ is the same for any Barannikov basis.
\end{thm}
\begin{rem}
  \label{rem:thm_classification}
  \begin{enumerate}[label=\arabic*.]
  \item Barannikov basis itself need not be unique (compare
    \Cref{rem:bases_not_unique}).
  \item Put differently, one may say that B-data is a full invariant of
    an enhanced complex considered up to isomorphism (see \Cref{subs:term_dig}).
  \item It is profitable to have a Barannikov basis at hand, since the complex
    takes the simplest form possible and becomes tractable.
  \item The case when the complex is not enhanced, but only filtered, was proven
    in \cite{Bar}. See also \Cite{Meln,Thij1,Thij2} for an ``ungraded'' setting
    where a single upper triangular matrix is considered.
  \end{enumerate}
\end{rem}

For a matrix $X$ we denote by $X_{\bullet, j}$ its $j^{\text{th}}$ column
and by $X_{i, \bullet}$ it's $i^{\text{th}}$ row.

\begin{proof}[Proof of \Cref{thm:classification}]
  Uniqueness follows from the existence and the fact that B-data is
  invariant under isomorphisms (which is shown in \Cref{subs:bd}). The rest is
  devoted to proving the existence part.

  Fix any $k \in \{0, \ldots, \dim \C\}$. The differential
  $\d_k:C_k \to C_{k-1}$ is a map of enhanced vector spaces $(C_k, \kappa_k)$
  and $(C_{k-1}, \kappa_{k-1})$ (see \Cref{rem:enh_c_k}). Therefore,
  \Cref{lem:class_space} produces, in particular, a rook matrix $X$ and a basis
  $x$ of $(C_k, \kappa_k)$. Analogously, applying \Cref{lem:class_space} to
  $\d_{k+1}$ one obtains a rook matrix $Y$ and another basis of
  $(C_k, \kappa_k)$, call it $y$. Construct now a third basis $v$ of
  $(C_k, \kappa_k)$ as follows.

  Fix $s \in \{ 1, \dots, \dim C_k \}$. Obviously at least one of three cases
  listed below holds. On the other hand, it follows from $\d^2 = 0$ and proof of
  \Cref{lem:class_space} that all three are mutually exclusive. So, we define
  $v_s$ depending on which of them holds.
  \begin{enumerate}[label=\arabic*)]
  \item $X_{\bullet, s} \neq 0$. Set $v_s \colonequals x_s$.
  \item $Y_{s, \bullet} \neq 0$. Set $v_s \colonequals y_s$.
  \item Both $X_{\bullet, s}$ and $Y_{s, \bullet}$ are zero. One is free to take
    either $x_s$ or $y_s$ as $v_s$.
  \end{enumerate}
  We have a constructed a basis of $(C_k, \kappa_k)$ for each $k$. The last step
  is to construct a basis $c$ of $(\C, \kappa)$. Take any
  $s \in \{ 1, \dots, \dim \C_\bullet \}$. Let $v$ be a constructed basis of
  $(\C_{\deg s}, \kappa_{\deg s})$. Define $c_s$ to be $v_{\dim C_{\deg s}^s}$
  (see \Cref{def:enh_complex}). Finally, by construction $c$ is a Barannikov
  basis.
\end{proof} 

\begin{rem}
  Three mentioned cases correspond respectively to the fact that $s$ belongs
  to \begin{enumerate*}[label=\arabic*)] \item U, \item L, \item
    H \end{enumerate*}.
\end{rem}

\subs{$\Z$-enhanced complexes} \label{subs:Zenh_comp} In this subsection we
introduce a certain analogue of an enhanced complex, which is itself a complex
of free abelian groups.

\begin{defin}
  \label{def:Zenh_comp}
Let $\C$ be a (chain) complex of free abelian groups
\[
 C_{n+1} = 0 \to C_n \xto{\d_n} \dots \xto{\d_1} C_0 \to 0 = C_{-1}.
\]
A \emph{$\Z$-enhancement $\kappa$ on a complex} $\C$ is a choice of the
following two structures.
\begin{enumerate}[label=\arabic*)]
\item A filtration
  \[
    0 = \C^0 \subset \ldots \subset \C^{\rk \C_\bullet} = \C
  \]
  of $\C$ by subcomplexes s.t. for each $s \in \{1, \dots, \rk \C_\bullet \}$
  the quotient complex $\sfrac{\C^s}{\C^{s-1}}$ is isomorphic to $\Z$
  concentrated in one degree.
  \item A generator of $\sfrac{\C^s}{\C^{s-1}} \simeq \Z$.
\end{enumerate}
A complex with a $\Z$-enhancement will be called a \emph{$\Z$-enhanced complex}
and denoted as $(\C,\kappa)$.
\end{defin}

The following notions and statements go in exactly the same manner as in the
honest enhanced case. 
\begin{enumerate}[label=\arabic*)]
\item The definition of an isomorphism between two $\Z$-enhanced complexes.
\item The definition of a basis of a $\Z$-enhanced complex (recall that by a
  basis we always mean a linearly ordered set of generators). Matrix of
  differential $\d_k$ in any basis is obviously integral, yet it will play
  important role in the end of this subsection.
\item Every $\Z$-enhanced complex can be equipped with a basis.
\item Let $c = (c_1, \dots, c_{\rk \C_\bullet})$ be a basis of a complex of free
  abelian groups s.t. \begin{enumerate*}[label=\arabic*)] \item for each $s$ the
    span $\langle c_1, \dots, c_s \rangle$ is a subcomplex, \item the induced
    filtration satisfies the condition 1) from
    \Cref{def:enh_complex}. \end{enumerate*} Then one can construct a
  $\Z$-enhanced complex $(\C, \kappa(c))$.
\end{enumerate}

It follows directly from the definitions that if $(\C,\kappa)$ is a
$\Z$-enhanced complex then $\C \otimes \F$ is an enhanced complex over $\F$. We
denote it by $(\C \otimes \F, \kappa)$.

\begin{rem}
  An oriented strong Morse function on a manifold naturally gives rise to a
  \mbox{$\Z$-enhanced} complex. However, classifying such complexes up to
  isomorphism is a transcendentally hard problem. So, following \cite{Bar}, we
  proceed by tensoring the given complex by $\F$ for various fields. See
  \Cref{subs:enh_comp_morse}.
\end{rem}

The rest of this subsection is devoted to the interplay between properties of
$\Z$-enhanced complex $(\C, \kappa)$ and those of enhanced complex
$(\C \otimes \Q, \kappa)$. So we set $\F = \Q$ for the time being. Similar in
spirit results are stated in \Cref{subs:q} without proofs.

Recall that B-data may be viewed as a sequence of rook matrices $\{ R_k \}$, see
\Cref{subs:bd} and \Cref{rem:bd_class_space}. Recall also that in
\Cref{subs:matr_spaces} we associated a subset $\Ll(R)$ of matrices with a rook
matrix $R$.

\begin{prop}
  \label{prop:matr_comp}
  Let $c$ be any basis of a $\Z$-enhanced complex $(\C, \kappa)$. Then the
  matrix of differential $\d_k$ in this basis belongs to the set $\Ll(R_k)$, where
  $R_k$ is a rook matrix from the B-data of $(\C \otimes \Q, \kappa)$. 
\end{prop}
\begin{proof}
  Let $D$ be a matrix of differential $\d_k$ in some basis $c$ of
  $(\C, \kappa)$. After choosing another basis $c'$ the matrix $D$ gets
  multiplied by a unitriangular matrices (over $\Z$) from the left and from the
  right. The statement now follows from \Cref{prop:matr_spaces}.
\end{proof}
See \Cref{subs:matr_spaces} for an example, which may be treated as a complex
concentrated in two degrees. By a degree of a pair we will mean degree of its
lower point. A Barannikov pair is called short if there are no pairs of the same
degree that lie inside it. Formally, $(s, t)$ is a short pair if there is no
pair $(s', t')$ of the same degree s.t. $s < s' < t' < t$.

\begin{cor}
  \label{cor:short_comp}
  Let $(\C, \kappa)$ be a $\Z$-enhanced complex. Bruhat number of
  $(\C \otimes \Q, \kappa)$ on any short pair is integer.
\end{cor}
\begin{proof}
  Take any short pair of degree, say, $k-1$. Let $R_k$ be a rook matrix from the
  B-data of $(\C \otimes \Q, \kappa)$. Short pairs correspond precisely to those
  non-zero entries of $R_k$ which are not covered (in the terminology of
  \Cref{subs:matr_spaces}). The statement now follows from
  \Cref{prop:matr_comp}.
\end{proof}

The next statement is a mere combination of the previous two. 
\begin{cor}
  \label{cor:short_matr_comp}
  Let $(\C, \kappa)$ be $\Z$-enhanced complex and let $(s,t)$ be a short
  Barannikov pair of $(\C \otimes \Q, \kappa)$. Let also $c$ be any basis of
  $(\C, \kappa)$. Then element $c_s$ appears in the differential of $c_t$ with
  the coefficient equal to Bruhat number on a pair $(s,t)$. In particular, this
  coefficient doesn't depend on $c$.
\end{cor}

\subs{Bruhat numbers over the rationals}
\label{subs:q}
In this subsection we state several
facts about interplay between $\Z$-enhanced complex $(\C, \kappa)$ and enhanced
complex $(\C \otimes \Q, \kappa)$. The proofs will be given elsewhere.

For an abelian group $G$ we denote by $\# G$ its order and by $\Tors G$ its
torsion subgroup.

\begin{prop}
  Let $(\C, \kappa)$ be a $\Z$-enhanced complex. Let also $s$ and $t$ ($s > t$,
  both from $\{ 1, \dots, \rk \C_\bullet \}$)
  be a Barannikov pair of enhanced complex $(\C \otimes \Q, \kappa)$ with Bruhat
  number $\lambda \in \Q^*$. One then has
\[
  \pm \lambda = \frac{\# \Tors \H_\bullet(\C^s, \C^{t-1})}{\# \Tors
    \H_\bullet(\C^{s-1}, \C^t)} = \frac{\# \Tors \H_{\deg t}(\C^s, \C^{t-1})}{\#
    \Tors \H_{\deg t}(\C^{s-1}, \C^t)}.
  \]
\end{prop}

In \Cref{sec:br_tors} will interpret Bruhat numbers as a certain kind of
torsion. From this viewpoint the given formula is of type ``torsion=torsion''.
For its close relative, see \cite[Theorem 4.7]{Tur}, proven in weaker generality
by Milnor \cite{MilTor}. See also \cite{Charette} for similar in spirit
statement in symplectic topology. We stress out that we place no acyclicity 
condition on a complex $\C$.

\begin{prop}
  Let $(\C, \kappa)$ be a $\Z$-enhanced complex. Then the following are
  equivalent:
  \begin{enumerate}[label=\arabic*)]
    \item $\Tors \H_\bullet(\C^s, \C^t) = 0$ for all $s > t$ (both from $\{ 1,
      \dots, \rk \C_\bullet \}$),
    \item all the Bruhat numbers of $(\C \otimes \Q, \kappa)$ equal to $\pm 1$,
    \item the $\Z$-enhanced complex $(\C, \kappa)$ is isomorphic (in the sense
      of item 1) in the list from \Cref{subs:Zenh_comp}) to the direct sum of
      complexes of two forms: \begin{enumerate*}[label=\arabic*)] \item
        $0 \to \Z \to 0$, and \item $0 \to \Z \xto{\sim} \Z \to
        0$. \end{enumerate*}
  \end{enumerate}
\end{prop}

\subs{Taking B-data commutes with taking sub- and quotient complexes} Given an
enhanced complex $(\C,\kappa)$ and two integers
$0 \leq l \leq m \leq \dim \C_\bullet$, consider a quotient complex
$\sfrac{\C^m}{\C^l}$; it inherits an enhancement. The goal of this subsection is
to provide a recipe on how to express B-data of this complex in terms
of the initial one. To this aim we, first, describe this recipe and, second,
prove that it is correct.

Let us fix the notations first. Aforementioned enhanced complex will be denoted
as $(\C|_l^m,\kappa|_l^m)$. Let $(U,L,H,b,\lambda)$ be a B-data associated to
$(\C,\kappa)$. Given $l$ and $m$ we will now define another B-data
$(U',L',H',b',\lambda')$.

Set $U' \colonequals \{s \in U| l < s \leq m, b(s) > l \} - l$ (by convention,
subtracting an integer $l$ from a subset of integers yields another subset
formed by differences with $l$ of each element individually),
$L' \colonequals \{s \in L| l < s \leq m, b^{-1}(s) \leq m \} - l$,
$H' \colonequals \{1, \dots, m-l \} \setminus (U' \sqcup L')$. Define the
grading on $U'$ via its injection into $U$. Proceed similarly for $L'$ and $H'$.
For $s \in U'$ define \begin{enumerate*}[label=\arabic*)] \item
  $b'(s) \colonequals b(s+l)-l$, \item $\lambda'(s) \colonequals
  \lambda(s+l)$. \end{enumerate*}

\begin{prop}
  Let $(\C,\kappa)$ be an enhanced complex and $(U,L,H,b,\lambda)$ be its
  B-data. For a given $0 \leq l \leq m \leq \dim \C_\bullet$ the B-data of
  $(\C|_l^m,\kappa|_l^m)$ coincides with the data $(U',L',H',b',\lambda')$
  constructed above.
\end{prop}
\begin{proof}
  Take Barannikov basis of $(\C,\kappa)$ which exists by
  \Cref{thm:classification}. Its elements with indices from $l+1$ to $m$, when
  mapped to a $(l,m)$-slice, again form a Barannikov basis. The statement
  follows.
\end{proof}
\begin{rem}
  Informally, the B-data of $(\C|_l^m,\kappa|_l^m)$ is obtained from that of
  $(\C, \kappa)$ by the simplest procedure possible: one has to cut it from
  below and above at the given levels $l$ and $m$.
\end{rem}
\begin{rem}
  Although usage of \Cref{thm:classification} makes the proof shorter, it is
  still possible to prove the above statement directly from the definitions.
\end{rem}

\subs{Torsion of a chain complex}
\label{subs:tors_enh}
We begin this subsection by recalling Milnor's \cite{MilTor} definition of
torsion of a chain complex, closely following \cite{Tur}. After this we define
torsion of an enhanced complex and calculate it in terms of B-data.

Let $v=(v_1, \dots, v_{\dim V})$ and $v'=(v'_1, \dots, v'_{\dim V})$ be two
bases of $V$. Denote by $[v'/v] \in \F^*$ the determinant of a transition matrix
from $v$ to $v'$. We call two bases $v$ and $v'$ equivalent if $[v'/v] = 1$. Let
us now be given an exact triple of vector spaces
$0 \to U \xto{\iota} V \xto{\pi} W \to 0$ along with bases $u$ and $w$ of $U$
and $W$ respectively. Construct a basis $uw$ of $V$ as follows. For a vector
$w_i \in W$ set $\pi^{-1}(w_i) \in V$ to be any lift w.r.t. $\pi$. Now set
$uw \colonequals (\iota(u_1), \dots, \iota(u_{\dim U}), \pi^{-1}(w_1), \dots,
\pi^{-1}(w_{\dim W}))$. %
Equivalence class of $uw$ is independent of chosen lifts of $w_i$'s.

Recall that for a chain complex $(\C,\d)$ one defines boundaries $B_k$ to be
$\im \d_{k+1}$ and cycles $Z_k$ to be $\ker \d_k$. One then has two exact
triples:
\begin{gather}
  0 \to \B_k \to \Zy_k \to \H_k \to 0, \label{torsion:eqn:first_triple}\\
  0 \to \Zy_k \to C_k \xto{\d_k} \B_{k-1} \to
  0 \label{torsion:eqn:second_triple}.
\end{gather}

Let us now be given bases $c_k$ of $C_k$ and $h_k$ of $\H_k$ (for all admissible
$k$). Choose any basis $b_k$ of $\B_k$. Construct, first, a basis $b_k h_k$ of
$\Zy_k$ via triple (\ref{torsion:eqn:first_triple}), and, second, a basis
$b_k h_k b_{k-1}$ of $C_k$ via triple (\ref{torsion:eqn:second_triple}). Define
the torsion of $\C$ to be
\[
\tau(\C)\colonequals \prod_{i=0}^{n}[b_ih_ib_{i-1}/c_i]^{(-1)^{i+1}}\in\F^*.
\]
It is straightforward to show that $\tau(\C)$ depends only on $(\C,\d)$,
equivalence class of $c_k$ and that of $h_k$ (see \cite{Tur}).

\begin{rem}
  \label{rem:tors_basis}
  If one replaces basis $c_k$ with $c_k'$ for some particular $k$ then the
  torsion gets multiplied by $[\sfrac{c_k'}{c_k}]^{(-1)^{k+1}}$.
\end{rem}

Any two bases of a given enhanced vector space $(V,\kappa)$ are equivalent.
Analogously, for any two bases $c$ and $c'$ of an enhanced complex $(\C,\kappa)$
one has equivalence between $\underline{c}_k$ and $\underline{c'}_k$, where by
$\underline{c}_k$ we mean (here and further) an ordered subset of $c$
corresponding to a basis of $(C_k, \kappa_k)$ (see \Cref{rem:enh_c_k}).

Let us now assemble all the pieces together. Let an enhanced complex
$(\C,\kappa)$ be given. Choose any basis $c$ of $(\C,\kappa)$. Recall that by
\Cref{subs:enhancement_on_h} we have an enhanced vector space
$(\H_k(\C),\kappa_{\H_k})$ for each $k$. Choose any basis $h_k$ of
$(\H_k(\C),\kappa_{\H_k})$. Define the torsion of an enhanced complex
$(\C,\kappa)$ to be the torsion of $\C$ w.r.t. bases $\underline{c}_k$ and
$h_k$; denote it by $\tau(\C,\kappa)$. This number is well-defined since
equivalence classes of both $\underline{c}_k$ and $h_k$ are well-defined.

\begin{rem}
  We stress out that $\tau(\C, \kappa)$ depends only on the enhancement on
  $\H_\bullet(\C)$. Bruhat numbers, however, depend on enhancement on
  $\H_\bullet(\C^s)$ for various $s$.
\end{rem}

\begin{rem}
  Some kind of interplay between filtration and torsion also appears in
  \cite[Appendix A]{GKZ}. 
\end{rem}

\begin{con}
  \label{con:perm}
  Let $(N,U,L,H,b)$ be a part of a B-data. We will now construct a
  permutation $\sigma$ of $N$ elements.
\end{con}  
\triright Note that $b$ doesn't have anything to do with $b_i$ from the
definition of torsion. For a fixed $k$, the set $U$ determines a subset of a set
$\{1, \dots, \# \{ s \in \{ 1, \dots, N \} | \deg s = k \} \}$, call it $U_k$.
Define $L_k$ and $H_k$ similarly; the map $b$ determines a bijection
$b_k : U_k \to L_{k-1}$. We will now define a permutation $\sigma_k$ on
$\# \{ s \in \{ 1, \dots, N \} | \deg s = k \}$ elements by writing integers in
a row. First, write down elements of $L_k$ in increasing order. Second, write
down elements of $H_k$ also in increasing order. Third, write down elements of
$U_k$, but this time in the order of increasing of $b_k(s)$, for $s \in U_k$.

For two permutations $\sigma$ and $\pi$ of length $l$ and $m$ their (direct) sum
$\sigma + \pi$ is defined as a permutation of $l + m$ elements acting as
$\sigma$ on the first $l$ elements and as $\pi$ on the last $m$ elements. We
define $\sigma$ to be the sum $\sigma_0 + \ldots + \sigma_n$, where
$n=\max_{s \in \{ 1, \dots, N \} } \deg s$. \trileft

The sign of a permutation $\sigma$ will be denoted as $(-1)^\sigma$.
\begin{prop}
  \label{prop:tors_enh}
  Let $(\C,\kappa)$ be an enhanced complex and $(U,\lambda)$ be a part of its
  B-data. Let also $\sigma$ be the permutation from \Cref{con:perm}. We
  then have
  \[
    \tau(\C,\kappa)=(-1)^\sigma \prod_{s \in U}\lambda(s)^{(-1)^{\deg s}}.
  \]
\end{prop}
\begin{proof}
  Since any two bases of an enhanced vector space are equivalent, we may
  calculate $\tau(\C,\kappa)$ in some Barannikov basis $c$, which exists by
  \Cref{thm:classification}. We continue using notations introduced in the
  beginning of this subsection. Choose basis $h_k$ to be $c_s$ for all
  $s \in H$. Similarly, choose basis $b_k$ to be $c_s$ for all $s \in L$ (the
  linear order on both bases is induced from that on $c$). These choices yield a
  right-hand side by the very definitions.
\end{proof}

\section{Morse theory}
\label{sec:morse}
 
In the first part of this section we introduce, after neccesary preparations, a
construction which associates an enhanced complex over $\F$ with a strong Morse
function (see \Cref{subs:enh_comp_morse}). This justifies a thorough study of
enhanced complexes in the previous section. We then proceed to discuss various
properties of Bruhat numbers of a given strong Morse function. The majority of
results translates readily to the setting of discrete Morse theory in a sense of
Forman \cite{Forman}; as in the smooth case, the strongness assumption on a
function is crucial and must be satisfied.

\subs{Setup}
\label{subs:morse_setup}
In this subsection we recall basic notions of Morse theory and fix appropriate
notations, setting the stage for our results.

Let $M$ be a smooth closed manifold fixed once and for all throughout this
section. Recall that a smooth function $f \colon M \to \R$ is called Morse if
all its critical points are non-degenerate. A smooth function is called strong
if all its critical values are pairwise distinct. Fix a strong Morse function
$f$ on $M$ once and for all throughout this section. For $a \in \R$ the subspace
$M^a \colonequals \{ x \in M | f(x) \leq a \}$ is called a sublevel set.

Morse' idea was to track how the homotopy type of $M^a$ changes while $a$ grows
from $-\infty$ \mbox{to $+\infty$}. This is perfomed by investigating the
critical points of $f$, the set of which is denoted by $\Cr(f) \subset M$. Since
$f$ is strong those are in bijection with critical values of $f$ (this set is
finite because of the compactness of $M$). Keeping this bijection in mind, we
will freely switch between points and values without mentioning this explicitly.
The set $\Cr(f)$ is $\Z_{\geq 0}$-graded by index of a critical point, the
degree of $c \in \Cr(f)$ is denoted by $\deg c$. Though it's more natural to say
``index of critical point'', we will mostly say ``degree'' in order to be
consistent with \Cref{sec:enh_comp}. The set of all critical points of degree
$k$ is denoted by $\Cr_k(f)$. Note that the set $\Cr(f)$ is also naturally
linearly ordered; we denote by $c_s \in \Cr(f)$ (for
$s \in \{1, \dots, \# \Cr(f) \}$) its $s^{\text{th}}$ element w.r.t. this order.
By $\eps$ we will mean a sufficiently small positive real number.

It follows from foundational results of Morse theory, which we recall in
\Cref{subs:cw_comp_morse} that \mbox{for $c \in \Cr(f)$} one has
$\H_{\deg c} (M^{f(c) + \eps}, M^{f(c) - \eps}; \Z) \simeq \Z$. We say that a
critical point is oriented if the generator of this free abelian group of rank
one is chosen. A strong Morse function is called oriented if all its critical
points are oriented (see \Cref{subs:orient} for a discussion).

It will be convenient to fix the set of real numbers $r_s$ (for
$s \in \{ 0, \dots, \# \Cr(f) \}$) s.t.
\[
  r_0 < f(c_1) < r_1 < \ldots < f(c_{\# \Cr(f)}) < r_{\# \Cr(f)}.
\]
Such numbers are called regular values. 

Fix a field $\F$ once and for all. All the chain complexes and homologies are
assumed to be over $\F$ unless stated otherwise. If the group of coefficients is
given explicitly, it goes after a semicolon, e.g. $\H_2(M; \Z)$.

\subs{B-data associated with a strong Morse
  function} \label{subs:bd_morse} In this subsection we present and discuss the
\begin{con}
  \label{con:bd_morse}
  Let $f$ be an oriented strong Morse function on $M$ and $\F$ be a field. We
  will now construct a B-data.
\end{con}  
\triright Since $\Cr(f)$ is linearly ordered, it is in natural bijection with
$\{ 1, \dots, N\}$; define the grading on the latter by that on the former. The
manifold $M$ is filtered, as a topological space, by subspaces
\[
  \emptyset = M^{r_0} \subset M^{r_1} \subset \ldots \subset M^{r_N} = M.
\]
Moreover, for each $ s \in \{ 1, \dots, N \}$ the space $M^{r_s}$ is homotopy
equivalent to $M^{r_{s-1}}$ with a single cell of dimension $\deg s$ attached.
Since $f$ is oriented, the one-dimensional vector space
$\H_{\deg s} (M^{r_s}, M^{r_{s-1}}) \simeq \H_{\deg s} (M^{r_s}, M^{r_{s-1}};
\Z) \otimes \F \simeq \F$ has a preferred basis $o \otimes 1$, where $o$ is a
generator of $\H_{\deg s} (M^{r_s}, M^{r_{s-1}}; \Z) \simeq \Z$. Therefore, the
complex of singular chains on $M$ (with coefficients in $\F$) satisfies
conditions of \Cref{con:enhancement_on_h} and we are able to extract a B-data as
in \Cref{subs:bd} (see \Cref{rem:bd_gen}). \trileft

In particular, we have just constructed enhancement on a homology vector space
$\H_\bullet(M)$ as well as on $\H_\bullet(M^{r_s})$ for all $s$. We call the
image of $\lambda$ ``Bruhat numbers'' of oriented Morse function; the same goes
for the Barannikov pairs (or, briefly, pairs). Informally, \Cref{con:bd_morse}
decomposes some critical points (equivalently, values) of $f$ into Barannikov
pairs. Moreover, it a associates a Bruhat number (i.e. an element of $\F^*$)
with each pair (see the beginning of the \Cref{subs:bd}). Points
$c_s \in \Cr(f)$ s.t. $s \in H$ are called homological critical points. The
number of homological points of index $k$ equals to $\dim \H_k(M)$ (see
\Cref{subs:bd}). Analogously points from $U$ (resp. $L$) are called upper (resp.
lower). We stress out that we haven't yet considered any finite-dimensional
approximation of filtered complex of singular chains on $M$; we will do so in
\Cref{subs:cw_comp_morse,subs:enh_comp_morse}.

\begin{rem}
  \label{rem:orient}
  Changing the orientation of some critical point $c_s \in \Cr(f)$ alters the
  B-data as follows. First, the decomposition into pairs stays the
  same. Second, if $c_s$ is homological, then the whole B-data stays
  the same. Otherwise, if $c_s$ belongs to some pair, then the Bruhat number on
  this pair gets multiplied by $-1$. Therefore, canonically we can associate
  Bruhat numbers to pairs only up to a sign.
\end{rem}

\begin{rem} 
  \label{rem:or_bar}
  The original Barannikov's construction \cite{Bar} produces the same set of
  pairs. To see this, one should combine \Cref{rem:thm_classification} and
  results from \Cref{subs:enh_comp_morse}. See also \cite{EH} for a topological
  data analysis perspective. 

  A close idea of construction of Bruhat numbers over $\Q$ appeared
  independently in \cite{Vit2}.
\end{rem}
    
We conclude this subsection by several remarks. The construction of homological
critical points goes back to Lyusternik and Shnirelman. Note that this
construction implies weak Morse inequalities:
$\# \Cr_k(f) \geq \dim \H_k(M; \F)$ for any field $\F$. See also \cite{Vit0} for
the innovative fruitful applications of similar ideas in symplectic topology.

\Cref{rem:4term} translates to topological setting as follows. The condition for
two critical values $a$ and $b$ to form a Barannikov pair is equivalent to \[
  \dim \H_\bullet(M^{a - \eps}, M^{b + \eps}) =
  \dim \H_\bullet(M^{a + \eps}, M^{b - \eps}) =
  \dim \H_\bullet(M^{a - \eps}, M^{b - \eps}) - 1 =
  \dim \H_\bullet(M^{a + \eps}, M^{b + \eps}) - 1.
\]

See \cite{PC} and also \cite{Vit1}.

\subs{A digression on oriented Morse functions} \label{subs:orient} In this
subsection we give an alternative definition of an oriented Morse function (see
\Cref{subs:morse_setup}).

Let $Q$ be a quadratic form on a vector space $V$ over $\R$. Consider a vector
subspace $L \subset V$ of maximal dimension (this dimension is called negative
inertia index of $Q$) such that the restriction $Q|_L$ is negative-definite.
It's easy to see that the space $\{ L \}$ of all such vector subspaces is
contractible. Therefore, since any covering of a contractible space is trivial,
the space of all such oriented subspaces has two contractible components. Note
that a particular choice of a component determines an orientation of $L$ (as a
vector space).

Now, given a Morse function $f$ and any its critical point $p$, one may consider
a vector space $T_p M$ and a Hessian $\text{Hess}_p(f)$ on it. Function $f$ is
called oriented if, for any critical $p$ one of two mentioned components is
chosen. One may check that this definition coincides with the one given in
\Cref{subs:morse_setup}.

\subs{CW-complex associated with a strong Morse
  function} \label{subs:cw_comp_morse} In this subsection we briefly recall the
classical results from Morse theory, following \cite{Mi1}, in order to fix the
notations needed further in \Cref{subs:enh_comp_morse}.

\begin{thm}
  \label{thm:main_thm_morse}
  Let $f$ be an oriented strong Morse function on a closed manifold $M$. Then
  there exists a CW-complex $K$ s.t. the following holds:
  \begin{enumerate}[label=\arabic*)]
  \item $M$ is simple homotopy equivalent to $K$,
  \item cells of $K$ are in bijection with critical points of $f$. Moreover,
    dimension of a cell equals to the index of a critical point.
  \end{enumerate}
\end{thm}

\begin{rem}
  The CW-complex $K$ need not be unique (see \Cref{con:cw_comp}), but its simple
  homotopy type obviously is. Roughly speaking, the purpose of
  \Cref{subs:enh_comp_morse} is to encode the information which can be extracted
  uniquely from $f$ in algebraic terms. Note that Morse theory originated before
  CW-complexes were invented, see \cite{MB}.
\end{rem}
\begin{rem}
  Orientations of cells in $K$ may be naturally chosen by invoking orientation
  of $f$, see \Cref{con:main_lem_morse}.
\end{rem}

\begin{rem}
\label{rem:simple_tors}
The fact that the mentioned homotopy equivalence (as well as all the others in
this subsection) is actually simple is folklore. Although not stated explicitly
in \cite{Mi1}, it follows readily from the proof given there. We will need this
fact in \Cref{sec:br_tors} for statements involving (Reidemeister) torsion. We
will denote general homotopy equivalence by $\simeq$ and, whenever we want to
emphasize that it is simple, we write $\she$.
\end{rem}

The key ingredient in the proof of \Cref{thm:main_thm_morse} is the following
construction (we continue using notations introduced in
\Cref{subs:morse_setup}). (For a topological space $X$, by $X \cup_\phi e^k$ we
mean $X$ with a $k$-cell attached along $\phi$.)
\begin{con}
\label{con:main_lem_morse}
For $s \in \{ 1, \dots, \#\Cr(f) \}$ let $r_s$ and $r_{s-1}$ be two
corresponding regular values of $f$ and let $k = \deg s$. We will now recall the
construction of
a continuous map $\phi: S^{k-1} \to M^{r_{s-1}}$ s.t.
$M^{r_s} \she M^{r_{s-1}} \cup_\phi e^k$.
\choicemantra
\end{con}
\begin{rem}
  We stress out that it's not claimed that the homotopy class of a map $\phi$
  satisfying the property that $M^{r_s} \she M^{r_{s-1}} \cup_\phi e^k$ is
  unique. This assertion is only true for the map $\phi$ constructed by the
  recipe given below.
\end{rem}
\triright We will only sketch the argument, for details see \cite{Mi1}. The
space $M^{r_{s-1}}$ is a smooth manifold with boundary $f^{-1}(r_{s-1})$. Choose
an antigradient-like vector field on $M$. Its flow produces a smooth map
$S^{k-1} \to f^{-1}(r_{s-1}) \subset M^{r_{s-1}}$, where the source is viewed as
a small sphere around critical point $c_s$. For a different gradient-like vector
field the resulting map will differ by an isotopy. This way one gets an
embedding $M^{r_{s-1}} \cup_\phi e^k \into M^{r_s}$. It is then shown to be a
simple homotopy equivalence.

Note that the sphere $S^{k-1}$ is naturally oriented since any maximal vector
subspace in $T_{c_s}M$ on which Hessian restricts as a negative-definite form is
oriented (see \Cref{subs:orient}). Therefore the cell $e^k$ is naturally
oriented too. \trileft

The next proposition is obvious.
\begin{prop}
  Although the map $\phi$ from \Cref{con:main_lem_morse} depends on some choices
  made along the way, its homotopy class is uniquely defined.
\end{prop}
In order to get to \Cref{thm:main_thm_morse} one then proceeds with

\begin{con}
\label{con:cw_comp}
Let $s, r_s, r_{s-1}$ and $k$ be as in \Cref{con:main_lem_morse}. Suppose that
$M^{r_{s-1}} \she K$, where $K$ is some CW-complex. We will now recall the
construction of a cellular map $\psi: S^{k-1} \to K$ s.t.
$M^{r_s} \she K \cup_\psi e^k$ (note that r.h.s. is again a CW-complex).
\choicemantra
\end{con}
\triright Apply cellular approximation theorem to the map $\phi$ from
\Cref{con:main_lem_morse}. \trileft

Again, the next proposition is obvious.
\begin{prop}
  Although the map $\phi$ from \Cref{con:cw_comp} depends on some choices made
  along the way, its homotopy class is uniquely defined.
\end{prop}

\subs{Enhanced complex associated with a strong Morse function}
\label{subs:enh_comp_morse} In this subsection we discuss the following two
statements.
\begin{con} 
  \label{con:enh_comp_morse}
  Let $f$ be an oriented strong Morse function on $M$ and $\F$ be a field. We
  will now construct an enhanced complex $(\C, \kappa)$. \choicemantra
\end{con}
\begin{prop}
  \label{prop:enh_comp_morse}
  Although \Cref{con:enh_comp_morse} depends on some choices made along the
  way, the isomorphism class of an enhanced complex $(\C, \kappa)$ is uniquely
  defined.
\end{prop}

We will first describe \Cref{con:enh_comp_morse}.

\triright Take any CW-complex $K$ constructed by the virtue of
\Cref{thm:main_thm_morse}. Its cells are naturally linearly ordered by the order
of critical values of $f$. Moreover, the first $s$ cells form a CW-subcomplex,
simple homotopy equivalent to $M^{r_s}$. Consider an algebraic complex $\C'$ (of
free abelian groups) associated with $K$. It has a preferred (ordered) basis $c$
given by cells (mind that they carry natural orientation since $f$ is oriented,
see \Cref{con:main_lem_morse}). Thus one has a $\Z$-enhanced complex
$(\C',\kappa(c))$.

The desired enhanced complex is now taken to be $(\C' \otimes \F,\kappa(c))$
(see \Cref{subs:bd,subs:Zenh_comp}). \trileft
\begin{rem}
\label{rem:choices_cw}
  \Cref{con:enh_comp_morse} produces not only an enhanced complex, but also its
  basis. The matrix of differential, however, does depend on the choices made
  (practically, a cellular approximation from \Cref{con:cw_comp}). From this
  viewpoint, \Cref{prop:enh_comp_morse} says that for any two choices, the
  corresponding matrices of differential are conjugate by an unitriangular (i.e.
  triangular with ones on the diagonal) base change. This can be pushed to the
  full proof, see \Cref{rem:cw_comp_expl}.
\end{rem}

\begin{proof}[Proof of \Cref{prop:enh_comp_morse}]
  By the classificational \Cref{thm:classification} it suffices to prove that
  B-data associated to $(\C, \kappa)$ is uniquely defined (see
  \Cref{rem:thm_classification}). We will do that by identifying it with a
  B-data constructed from $f$ in \Cref{subs:bd_morse}. We denote the complex of
  singular chains \mbox{(over $\F$)} mentioned there by $\Csing$ and by
  $\Csing^s$ we mean the subcomplex corresponding to subspace $M^{r_s}$. Recall
  that since $f$ is oriented the generator of
  $\H_{\deg s}(\Csing^s, \Csing^{s-1}) \simeq \F$ is chosen (see
  \Cref{subs:bd_morse}).

  First of all, the two gradings on the set $\{ 1, \dots, \# \Cr(f) \}$ coincide
  since they are defined in terms of indices of critical points of $f$. Next,
  for each $s$ homology vector spaces $\H_k(\C^s)$ and $\H_k(\Csing^s)$ are
  naturally isomorphic (for all $k$) since both of them compute $\H_k(M^{r_s})$.
  Moreover, their filtrations are the same (w.r.t. the mentioned isomorphism)
  since they are defined topologically in terms of inclusions
  $M^{r_t} \into M^{r_s}$ for various $t$. Finally, the chosen generators in
  $\H_{\deg s}(\Csing^s, \Csing^{s-1})$ and in $\H_{\deg s}(\C^s, \C^{s-1})$ are
  the same for the same reason: they are defined topologically as a certain
  element in $\H_{\deg s}(M^{r_s}, M^{r_{s-1}})$. This implies that enhancements
  on $\H_k(\C^s)$ and $\H_k(\Csing^s)$ are the same. The statement now follows,
  since B-data was defined in \Cref{subs:bd} in terms of enhancements \mbox{on
    $\H_k(\C^s)$} for various $k$ and $s$.
\end{proof}

\begin{rem}
  \label{rem:q_fp}
  It follows from \Cref{subs:cw_comp_morse} that B-data associated
  $(\C, \kappa)$ coincides with the one constructed in \Cref{subs:bd_morse}. It
  follows from \Cref{rem:field_ext} that it is enough to consider only
  \mbox{fields $\Q$} and $\F_p$ (note that this is not the case in
  \Cref{subs:rt_recall,subs:rt_bruhat}).
\end{rem}    
 
\begin{rem}
  \label{rem:cw_comp_expl}
  It's possible to prove \Cref{prop:enh_comp_morse} somewhat more explicitly
  without appealing to \Cref{thm:classification}. We will now sketch the
  argument. Consider two approximations $\psi_1$ and $\psi_2$ from
  \Cref{con:cw_comp}. Since they are homotopic, corresponding algebraic
  complexes associated to CW-complexes differ by a change of basis. This change
  of basis is unitriangular by construction. Arguing inductively on the number
  of cells, one obtains a unitriangular change of basis which turns one chain
  complex into another. This means precisely that two correspoding enhanced
  complexes are isomorphic.
\end{rem}
  
\begin{rem}
  \label{rem:diff_s_t}
  B-data stays unaltered if one replaces $f$ with
  $\phi \circ f \circ \psi$, where $\psi \colon M \to M$ is any diffeomorphism
  and $\phi \colon \R \to \R$ is a diffeomorphism preserving an orientation.
  Consequently, B-data stays the same under the continuous deformation
  of $f$ in the class of oriented strong Morse functions on $M$ (see
  \Cref{subs:gens_1p}).
\end{rem}

\subs{Invariant of a map between two manifolds equipped with Morse functions}
\label{subs:two}
In this subsection we discuss the following construction.

\begin{con}
  Let $M_1$ and $M_2$ be two manifolds equipped with oriented strong Morse
  functions $f_1$ and $f_2$ respectively. Let also $l \colon M_1 \to M_2$ be a
  continuous map. We will now construct a rook matrix $R_k$ of size
  $\dim \H_k(M_2) \times \dim \H_k(M_1)$ (for any $k$).
\end{con}
\triright By \Cref{subs:bd_morse} the function $f_1$ (resp. $f_2$) gives rise to
an enhancement on a homology vector space $\H_k(M_1)$ (resp. $\H_k(M_2)$).
Consider the induced map $l_k \colon \H_k(M_1) \to \H_k(M_2)$ and plug it into
\Cref{lem:class_space} to get the desired rook matrix. \trileft

Without choosing a particular orientations of strong Morse functions the
non-zero entries of $R_k$ are defined only up to a sign. Note that we didn't
make use of any complexes whatsoever.

\subs{Morse complex} \label{subs:morse_comp} In this subsection we describe how
Morse complex fits into our setting of enhanced complexes. In particular, we
give a certain description of a matrix of Morse differential (w.r.t. any
Riemannian metric) in terms of B-data.

Let $g$ be a generic Riemannian metric on $M$ and $f$ be an oriented strong
Morse function. Then one can define a Morse complex $\M(f, g)$ whose integral
homology is naturally isomorphic to that of $M$. It's a complex of free abelian
groups, formally generated by critical points. In this basis the differential
(mapping $k$-chains to $(k-1)$-chains) becomes a matrix $(\d_{i,j})$. The matrix
element $\d_{i,j}$ equals to the number of antigradient flowlines from
$j^\text{th}$ critical point of index $k$ to $i^\text{th}$ critical point of
index $k-1$, counted with appropriate signs. For brevity, we say that
$i^{\text{th}}$ critical point ``appears in the differential'' of $j^\text{th}$
critical point with coefficient $\d_{i,j}$.

\begin{rem}
  As we saw in \Cref{subs:cw_comp_morse} in order to construct a complex
  generated by critical points out of a function $f$ one has to make a choice of
  cellular approximation. On the other hand, in order to construct a Morse
  complex one has to choose a Riemannian metric. These two choices are actually
  very close to each other in the following sense.

  Given a function $f$ one can construct a handle decomposition s.t. each handle
  of index $k$ corresponds to some critical point of index $k$. The metric $g$
  specifies the way to modify this decomposition, such that each handle of index
  $k$ is attached to the union of handles of smaller index. This, in turn,
  allows one define the handle complex, which is precisely the Morse complex. On
  the other hand, one can collapse handles to cells and obtain a CW-complex.
\end{rem}

In our terms, one has obtained a $\Z$-enhanced complex $(\M(f,g), \kappa(c))$,
where $c$ is a basis consisting of critical points (see \Cref{subs:Zenh_comp}).
It follows directly that B-data of $(\M(f,g) \otimes \F, \kappa(c))$
coincides with that of $f$ constructed in \Cref{subs:bd_morse}. Therefore, by
\Cref{thm:classification}, enhanced complex $(\M(f,g) \otimes \F, \kappa(c))$ is
isomorphic to $(\C, \kappa)$ from \Cref{subs:enh_comp_morse}. So, practically,
one may use any option to extract B-data from a given $f$: either a
complex of singular chains, or a CW-complex or a Morse complex.

\begin{rem}
  Analogously to \Cref{rem:choices_cw}, the matrix of Morse differential does
  depend on Riemannian metric, while the isomorphism class of enhanced complex
  $(\M(f,g) \otimes \F, \kappa(c))$ doesn't. The same goes for $\Z$-enhanced
  \mbox{complex $(\M(f,g), \kappa(c))$}.
\end{rem}

For a B-data associated to $f$ let $R_k$ be the corresponding rook matrix of
size $\Cr_{k-1}(f) \times \Cr_k(f)$ (see \Cref{subs:bd}). In other words,
non-zero elements of $R_k$ equal to the Bruhat numbers on Barannikov pairs of
points of degrees $k$ and $k-1$. The next theorem follows readily from
\Cref{prop:matr_comp}.

\begin{thm}
  \label{thm:matr}
  Let $f$ be an oriented strong Morse function on a manifold $M$. Let also $R_k$
  be the rook matrix associated to $f$ over $\Q$ (for
  $k \in \{ 1, \dots, \dim M \}$). Then the matrix of Morse differential $\d_k$
  w.r.t. any Riemannian metric $g$ belongs to the set $\Ll(R_k)$.
\end{thm}

For example suppose that $f$ has a B-data as depicted in
\Cref{fig:b_data_example} and $k = 2$. Then the corresponding rook matrix and
general form of a matrix of a second Morse differential $P$ are
\[
  R_2 =
\begin{pmatrix}
    0 & 0 & 4 \\
    3 & 0 & 0 \\
    0 & 0 & 0 \\
    0 & 2 & 0 \\
\end{pmatrix}
, P =
\begin{pmatrix}
    * & * & * \\
    3 & * & * \\
    0 & * & * \\
    0 & 2 & * \\
\end{pmatrix}.
\]

Recall that weak Morse inequalities state that
$\# \Cr_k (f) \geq \dim \H_k(M; \F)$ (for any field $\F$). It is easy to see
that if $\# \Cr(f) = \sum_k \dim \H_k(M; \Q)$ then the Morse differential
(w.r.t. any metric) must be identically zero. The next corollary is applicable
when this is not the case.

\CorShortMatrRestate*

\begin{proof}
  By assumption there is at least one Barannikov pair of $f$ over $\Q$. Take any
  short one. The statement now follows from \Cref{cor:short_matr_comp}.
\end{proof}

\begin{rem}
  The first part of \Cref{cor:short_matr} can be proven without appealing to any
  Barannikov pairs and Bruhat numbers whatsoever. Indeed, if one unwraps all the
  definitions and constructions involved, they arrive at a proof which is based
  on techniques like exact sequences of a pair. 
\end{rem}

\subs{A few examples and properties}
\label{subs:ex_props}
In this subsection we quickly give several introductory examples and properties
of Bruhat numbers.
\begin{wrapfigure}{r}{2cm}
  \includegraphics[width=1cm]{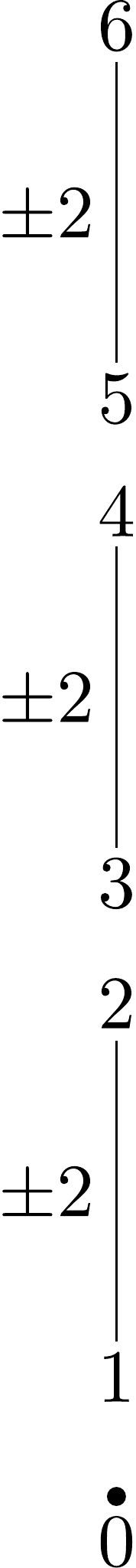}
  \caption{}
  \label{fig:RPn}
\end{wrapfigure}
\begin{enumerate}[label=
  \arabic*., wide, labelwidth=!, labelindent=0pt]
\item Let $f$ be a function on $\R\P^n$ which is a decent of
  $x_1^2 + 2x^2_2 + \ldots +(n+1)x_{n+1}^2$ defined on a unit sphere
  $S^n \subset \R^{n+1}$. It has $(n+1)$ critical points of all possible indices
  from 0 to $n$ (ordered by increasing of index). If $\ch \F = 2$ then all of
  them are homological. Otherwise, $(2k)^\text{th}$ and $(2k-1)^\text{th}$
  critical points form a Barannikov pair with Bruhat number $\pm 2$ (for any
  $k \in \{1, \dots, [n/2] \}$, where brackets denote the integral part). This
  is seen readily from any description of B-data, given either in
  \Cref{subs:bd_morse} or in \Cref{subs:enh_comp_morse}. See \Cref{fig:RPn} for
  an example for $n=6$.  
 
\item The proof of the next statement uses one-parameter Morse theory and
  therefore postponed till \Cref{subs:realiz}.
\begin{restatable}{prop}{PropRealizRestate}
  \label{prop:realiz}
  Let $\F$ be either $\Q$ or $\F_p$ and $\lambda \in \F^*$ be any non-zero
  number. Let also $M$ be any closed manifold s.t. $\dim M \geq 4$. Then one can 
  find an oriented strong Morse function $f$ on $M$ which has $\lambda$ as one
  of its Bruhat numbers.
\end{restatable}
In particular, Bruhat number over $\F = \Q$ may well be non-integer.

\item Recall that a Barannikov pair is called short if there are no pairs of the
  same degree lying inside it, see \Cref{subs:Zenh_comp}. The next statement
  follows directly from \Cref{cor:short_comp}.
  \begin{prop}  
    Over $\F = \Q$ Bruhat number on any short pair is integer.
  \end{prop}
  In particular, if there are Barannikov pairs over $\Q$ at all, at least one of
  them must carry integer Bruhat number.

\item The first statement from \Cref{subs:q} translates straightforwardly to the
  topological setting via \Cref{subs:enh_comp_morse}. In this setting it
  expresses Bruhat numbers of oriented strong Morse function over $\Q$ in terms
  of torsion in (relative) integral homology of various sublevel sets.
\end{enumerate}

\subs{Poincare duality}
\label{subs:PD}
In this subsection prove the following

\begin{prop}
  \label{prop:PD}
  Let $M$ be closed and orientable and $f$ be an oriented strong Morse function
  on it. Let also $\F$ be a field. Then B-data for $-f$ is B-data for $f$ turned
  upside down. Bruhat numbers on pairs remain the same.
\end{prop}

We need to first make some comments on the formulation. Since $M$ is orientable
and $f$ is oriented, the strong Morse function $-f$ is also naturally oriented.
Each critical point of $f$ of index $k$ is also a critical point of $-f$ of
index $n-k$, where $n = \dim M$. By turning the B-data upside down we
mean, formally speaking, precomposing all its ingredients with the automorphism
of the set $\{1, \dots, \# \Cr(f) \}$ given by $s \mapsto \# \Cr(f) - s$. Under
this operation upper and lower critical points swap their roles (while the
pairing remains the same). Note that the classical Poincare duality over the
field $\dim \H_k(M; \F) = \dim \H_{n-k}(M; \F)$ follows immediately from
\Cref{prop:PD}. Indeed, homological points remain homological after the
involution, but their indices change to complementary ones. The proof, however,
goes along the classical lines.

\begin{proof}[Proof of \Cref{prop:PD}]
  Choose a generic Riemannian metric $g$ on $M$. Matrix of differential in a
  Morse complex $\M(-f, g)$ is obtained from that in $\M(f, g)$ by
  transposition. Therefore, the same goes for the matrices of differential in
  the two corresponding based enhanced complexes. Consider now any unitriangular
  matrix $P_k$ which maps the basis $c$ (of critical points) of the first
  complex to a Barannikov basis. If follows from the formula
  $(P_{k-1} D_k P_k)^T = P_k^T D_k^T P_{k-1}^T$ that transposed matrix $P_k^T$
  (which is unitriangular w.r.t. reversed order) maps the initial basis of the
  second complex to the Barannikov one. Moreover, the matrices of differentials
  after these two changes of bases still differ by a transposition. The
  statement follows.
\end{proof}

\subs{On pairs of extremal degrees} \label{subs:pair_ext} Recall that by a
degree of a pair we mean degree of its lower point. In this subsection we prove

\begin{prop}
  \label{prop:pair_ext}
  Let $f$ be a strong Morse function on $M$ and $\F$ be a field.
  \begin{enumerate}[label=\arabic*)]
  \item The set of pairs of degree 0 is independent of $\F$. Bruhat number on any
    such pair is $\pm 1$.
  \item Suppose that $M$ is orientable. Then the set of pairs of degree
    $\dim M - 1$ is again independent of $\F$. Bruhat number on any such pair is
    again $\pm 1$.
    \end{enumerate}
\end{prop}
\begin{proof}
  Fix any $s \in \{ 1, \dots, \# \Cr(f) \}$ s.t. $c_s \in U$ (i.e. $c_s$ is an
  upper point in a pair) and $\deg s = 1$. We need to prove that
  $\lambda(s) = \pm 1$. Recall that $\{ r_s \}$ is a set of regular values of
  $f$ and let $d =\dim \H_0(M^{r_{s-1}})$. Let also $i_1 < \ldots < i_d$ be the
  ordered sequence of numbers which comprise the set
  $\{ t \in H^{s-1} | \deg t = 0 \}$, where $H^{s-1}$ is the set $H$ from the
  Brannikov data for the function $f|_{M^{r_{s-1}}}$. In other words,
  $(c_{i_1}, \ldots, c_{i_d})$ is the ordered set of homological points of
  degree zero for $f|_{M^{r_{s-1}}}$. It follows from
  \Cref{con:enhancement_on_h} that the sequence $([c_{i_1}], \ldots, [c_{i_d}])$
  is a basis (up to signs, which are irrelevant to the statement) of an enhanced
  vector space $\H_0(M^{r_{s-1}})$ (square brackets denote taking the homology
  class).

  On the other hand, the connecting homomorphism
  $\delta \colon \H_1(M^{r_s}, M^{r_{s-1}}) \to \H_0(M^{r_{s-1}})$ maps the
  chosen generator, represented by an oriented segment, to $[a] - [b]$, where
  $a$ and $b$ are the endpoints. The first statement now follows. The second one
  is obtained from it via the Poincare duality (\Cref{prop:PD}).
\end{proof}

\begin{rem}
  Orientability assumption in the second statement of \Cref{prop:pair_ext} is
  crucial. Indeed, the conclusion fails already for $M = \R\P^2$ and $\F = \Q$,
  see \Cref{subs:ex_props}.
\end{rem}

\begin{rem}
  Since the set of pairs from \Cref{prop:pair_ext} is independent of the field
  $\F$ one is tempted to find an alternative definition which doesn't involve
  $\F$. Indeed, it can be proven that the mentioned set may be recovered from
  the Kronrod-Reeb \cite{Kronrod,Reeb} graph of $f$ (see, for example,
  \cite{distReeb}). Moreover, the Kronrod-Reeb graph ``sees more'' than the set
  of pairs of extremal indices in a sense that one can find two functions on the
  same manifold with different Kronrod-Reeb graphs, but identical mentioned sets
  of pairs. The reason is, roughly, that B-data keeps track of filtration on
  homology induced by $f$, while Kronrod-Reeb graph keeps track of a canonical
  basis in $\H_0$ (or $\H_{\dim M}$) of the sublevel set.
\end{rem}

\section{Bruhat numbers and the theory of torsions}
\label{sec:br_tors}
As we saw in \Cref{prop:realiz} any number may appear as a Bruhat number of some
function; in a sense, there is no control over the individual Bruhat number.
However, sometimes the alternating product of all these numbers turns out to be
independent of $f$. Thus this product depends only on the manifold $M$. In the
present section we make this statement precise (in \Cref{subs:tors}) and provide
a framework where the mentioned product of Bruhat numbers equals to the
Reidemeister torsion of $M$ (in \Cref{subs:rt_bruhat}).  

\subs{Torsion of a Morse function}
\label{subs:tors}
In this subsection we set the stage for the further results.

\begin{defin}
\label{def:tors}
  Let $f$ be an oriented strong Morse function on $M$ and $\F$ be a field. Let
  also $\sigma$ be a permutation from \Cref{con:perm} (and $(-1)^\sigma$ its
  sign). The number
\[
  \tau(f, \F)=(-1)^\sigma \prod_{s \in U}\lambda(s)^{(-1)^{\deg s}} \in \F^*
\]
is called the \emph{torsion} of $f$ over $\F$.
\end{defin}

We refer to the r.h.s. as ``alternating product'' of all Bruhat numbers, in
analogy with alternating sum, which is used to define Euler characteristic.
Actually, torsion is very much similar to the Euler characteristic, see
\cite{Mnev}. We will simply write $\tau$ when its ingredients are understood. We
will now link \Cref{def:tors} with the classical notion of torsion (see a broad
yet concise book \cite{Tur} for an in-depth discussion).

By results from \Cref{subs:cw_comp_morse} starting from $f$ one can construct a
(non-unique) CW-complex $X$ which is simple-homotopy equivalent to $M$. So, from
the viewpoint of torsion theory $X$ and $M$ are the same (see
\Cref{subs:rt_recall}). Next, in \Cref{subs:enh_comp_morse} we used $X$ to
construct a (unique up to isomorphism) enhanced complex $(\C, \kappa)$. Further,
in \Cref{subs:tors_enh} we studied the torsion (defined in a classical way) of
any enhanced complex. By the very definitions it is a topological torsion of $X$
(and, therefore, $M$) w.r.t. to a certain basis in homology (see
\Cref{rem:tors}). Finally, \Cref{prop:tors_enh} identified the torsion (defined
in a classical way) with the formula from \Cref{def:tors}.

\begin{rem}
\label{rem:tors}
Here is a more concise way of saying the same. Take enhancement on a homology
$\H_k(M)$ given by $f$. Take any basis of this enhancement. Any two such bases
differ by a unitriangular matrix. Therefore, they are equivalent (in the sense
of \Cref{subs:tors_enh}). Thus the torsion of $M$ w.r.t. any chosen basis of
enhancement on $\H_k(M)$ is the same. This is precisely $\tau$.
\end{rem}

\begin{rem}
  In this section we will actually be interested only in the number
  $\pm \tau \in \sfrac{\F^*}{\pm 1}$. So the reader may temporarily disregard
  the permutation $\sigma$ and orientation of $f$. The sign will be important
  later in \Cref{sec:one_par}.
\end{rem}

\ThmNoHomRestate*

The proof requires bifurcation analysis and therefore postponed till
\Cref{subs:no_hom}. For example, taking $M$ to be $\R\P^n$ one sees that
$\tau(f,\Q) = \pm 2^{[n/2]}$, where brackets denote integral part. Indeed, one
has to calculate such a $\tau$ for some particular Morse function on $\R\P^n$.
They do so for a standart one from \Cref{subs:ex_props}.

\begin{rem}
  Since in \Cref{thm:no_hom} the number $\pm \tau(f, \F)$ turns out to be
  independent of $f$ one is tempted to give an alternative definition which
  doesn't involve $f$. We will now sketch the construction. Recall that the
  torsion is defined whenever there is a chosen basis in homology (or, at least,
  an equivalence class of such a basis). Usually in topology there's no
  canonical choice of such a basis, so one studies torsion in the setting where
  there's no homology at all. However, in $\H_0(M)$ and $\H_{\dim M}(M)$ there
  is an obvious distinguished basis. The number $\pm \tau(f, \F)$ is the torsion
  w.r.t. to it.
\end{rem}

\subs{Reidemeister torsion: recallment}
\label{subs:rt_recall}
In this subsection we briefly recall the notion of Reidemeister torsion (see
\cite{Tur} for details).

For a topological space $X$ let $\pi$ be its fundamental group and let
$\tilde{X} \to X$ be a universal covering. Choose a CW-decomposition of $X$
and denote the corresponding algebraic complex of free abelian groups by $\C$.
Consider the lift of the CW-structure to $\tilde{X}$. By the virtue of an
action of $\pi$ on its cells one construct a complex of free (right)
$\Z[\pi]$-modules $\tilde{\C}$ (recall that $\Z[\pi]$ is an integral group
ring of $\pi$). The rank of $k$-chains still equals to the number of $k$-cells
in $X$. Moreover, one can choose a basis in $\tilde{\C}$ by choosing any
lift of each cell. Therefore, each element of this basis is defined up to
multiplication by elements of $\pi$.

Let now $\rho \colon \Z[\pi] \to \F$ be a ring homomorphism. This gives $\F$ the
structure of a left $\Z[\pi]$-module by the formula $r \cdot x = \rho(r)x$,
where $r \in \Z[\pi], x \in \F$. On the other hand, $\F$ is also a right module
over itself. Putting all this together one now considers
$\tilde{\C} \otimes_\rho \F$, which is a complex of vector spaces over $\F$.
Moreover, this complex carries a basis, each element of which is defined up to
multiplication by elements of $\rho(\pi)$ (which is a multiplicative subgroup of
$\F^*$). The homology of $\tilde{\C} \otimes_\rho \F$ is called $\rho$-twisted
homology of $X$ and denoted here by $\H_*(X; \rho)$. If one is given an
equivalence class of its basis (in a sense of \Cref{subs:tors_enh}) then they
can consider torsion $\tau(X, \rho)$ of the complex of twisted chains, which
lives in $\sfrac{\F^*}{\rho(\pm \pi)}$. The sign ambiguity is due to several
things: \begin{enumerate*}[label=\arabic*)] \item there is no canonical
  orientation of cells of $X$, \item there is also no canonical linear order on
  them, \item there is no canonical CW-decomposition after all.\end{enumerate*}
If twisted homology vanishes, $\tau(X, \rho)$ is called the Reidemeister torsion
of $X$ w.r.t. $\rho$. It is known to be independent of CW-decomposition and to
be stable under simple homotopy equivalences. The theorem of Chapman
\cite{Chapman} states that homeomorphism is a simple homotopy equivalence.

\begin{rem}
  Practically, passing from $\tilde{\C}$ to $\tilde{\C} \otimes_\rho \F$
  means the following. Write the differential of $\tilde{\C}$ in some basis
  (for example, in the basis of lifts of cells) to get a matrix with
  coefficients in $\Z[\pi]$. Replace each coefficient $x$ with $\rho(x)$ and
  consider the resulting matrix as a map between vector spaces over $\F$.
\end{rem} 

\begin{rem}
  Traditionally torsion theory is said to be born in the beginning of
  $20^\text{th}$ century by the virtue of works of Reidemeister and Franz. We
  note that, quite surprisingly, the work of Cayley \cite{Cayley} contains some
  of the first torsion-theoretic ideas while staying in the purely algebraic
  setting.
\end{rem}

\subs{Reidemeister torsion and Bruhat numbers}
\label{subs:rt_bruhat}
In this subsection we introduce the notion of a twisted B-data and show
that the alternating product of its Bruhat numbers equals to the Reidemeister
torsion, whenever the latter is defined (see \Cref{thm:rt}). In particular, this
product is independent of a function.

Let $G$ be a multiplicative subgroup of $\F^*$ and $V$ be a vector space over
$\F$. By $\sfrac{V}{G}$ we will denote a set, which is a quotient of $V$ by the
natural action of $G$.

\begin{defin}
  An \emph{enhancement up to $G$ on a vector space} $V$ is a choice of two
  structures:
\begin{enumerate}[label=\arabic*)]
\item a full flag on $V$, i.e. a sequence of subspaces
  $0=V^{0}\subset V^{1}\subset \ldots \subset V^{\dim V}=V$ s.t.
  $\dim(\sfrac{V^{s}}{V^{s-1}})=1$, $s \in \{ 1, \dots, \dim V \}$;
\item a non-zero element $\kappa_{s}$ in a quotient set
  $\sfrac{(\sfrac{V^{s}}{V^{s-1}})}{G}$, $s \in \{ 1, \dots, \dim V \}$.
\end{enumerate}
Enhancement up to $G$ is still denoted by $(V,\kappa)$.
\end{defin}

Definitions of isomorphism between two such vector spaces as well of the complex
enhanced up to $G$ go exactly in the same manner as in the usual case. Moreover,
all the major statements from \Cref{sec:enh_spaces,sec:enh_comp} translate
readily to this new setting, with the only following exception. The non-zero
elements of rook matrix from \Cref{sec:enh_spaces} are only defined up to
multiplication by elements from $G$. Consequently, Bruhat numbers of a complex
enhanced up to $G$ now live in the quotient set $\sfrac{\F^*}{G}$.

If $f$ is a strong Morse function and $\F$ is a field of characteristic not two,
then $f$ defines a complex enhanced up to $\Z_2 = \{ \pm 1 \} \subset \F^*$ (see
\Cref{rem:orient}).

\begin{con}
  \label{con:enh_up_to}
  Let $f$ be an oriented strong Morse function on a manifold $M$, $\F$ be a
  field and $\rho \colon \Z[\pi] \to \F$ be a homomorphism of rings. We will now
  construct an isomorphism class of a complex enhanced up to $\rho(\pi)$.
\end{con}
\triright Apply \Cref{con:cw_comp} to $f$ to get a CW-complex $X$, simple
homotopy equivalent to $M$. Lift the CW-structure to the universal cover
$\tilde{X}$, choose any preimage of each cell and consider the corresponding
algebraic complex $\tilde{\C}$ of free $\Z[\pi]$-modules. It's basis is defined
up to action of $\pi$ and naturally linearly ordered. Moreover, the matrix of
differential is upper triangular w.r.t. this order, since so is differential in
$\C$. In particular, the span of the first $s$ basis elements is a subcomplex
(for $s \in \{1, \dots, \# \Cr(f) \}$).

Consider now the complex $\tilde{\C} \otimes_\rho \F$. It inherits a linearly
ordered basis which is defined up to an action of $\rho(\pi)$. The desired
complex enhanced up to $\rho(\pi)$ is now taken to be the one associated with
this basis (see the end of the \Cref{subs:bd}).

It remains to prove that obtained complex is well-defined up to an isomorphism.
Indeed, let $X'$ be another CW-complex obtained by the virtue of
\Cref{con:cw_comp}. Matrices of cellular differentials (in any degree) in these
two complexes are conjugated by a unitriangular matrix \mbox{over $\Z$} (see
\Cref{rem:cw_comp_expl}). Therefore, the matrices of differentials in
$\tilde{\C}$ and $\tilde{\C'}$ are conjugated by a triangular matrix with
elements from $\pi$ on the diagonal. The desired statement follows. \trileft

Consequently, given a data as in the \Cref{con:enh_up_to} one can construct
Barannikov pairs and Bruhat numbers, which are elements of
$\sfrac{\F^*}{\rho(\pi)}$ (without choosing a particular orientation of $f$
these numbers live in $\sfrac{\F^*}{\rho(\pm \pi)}$). To emphasize the presence
of $\rho$ we say ``twisted Barannikov pairs'' and ``twisted Bruhat numbers''.
One then defines torsion $\tau(f, \rho)$ of $f$ exactly as in \Cref{def:tors}.
Again, as in \Cref{subs:rt_recall}, \Cref{prop:tors_enh} justifies the name. In
short, the alternating product of twisted Bruhat numbers of $f$ equals to the
torsion of $M$ w.r.t. to a certain basis of the vector space
$\H_\bullet(M; \rho)$ defined by $f$. Generally, this basis and, consequently,
$\tau(f, \rho)$ may well depend on $f$; this basis is even non-uniquely defined,
but this arbitrariness doesn't affect $\tau(f, \rho)$.

However, combining all of the above with \Cref{subs:rt_recall} one gets

\ThmRTRestate*

\section{One-parameter Morse theory}
\label{sec:one_par}

One-parameter Morse theory deals with generic paths (in other words,
one-parameter families) in the space of all smooth functions on $M$. The
endpoints of a generic path are strong Morse functions~--- this is essentially
the statement that strong Morse functions form an open dense subspace. However,
finitely many points of such a path may fail to be either strong or Morse
functions. It is exactly at these points where the B-data associated with a
strong Morse function changes. In this section we describe how exactly these
changes look like (see \Cref{subs:bd_families}). This allows to prove some
statements from the previous sections (see \Cref{subs:no_hom}). On the other
hand, this also enables us to reprove a theorem of Akhmetev-Cencelj-Repovs
\cite{Akhm} in greater generality (see \Cref{subs:akh}).

\subs{Generalities on one-parameter Morse theory} \label{subs:gens_1p} In this
subsection we recall foundations of one-parameter Morse theory, initiated by
Cerf \cite{Cerf}.

A path in the space of functions on $M$ is practically a map
$F \colon M \times [-1, 1] \to \R$. Let $t$ be a coordinate along $[-1, 1]$,
which we occasionally refer to as ``time''. Define the function
$f_t \colon M \to \R$ by $f_t(x) \colonequals F(x, t)$. For convenience we will
write $\{ f_t \}$ instead of $F$. By a point of a path $\{f_t \}$ we will mean a
function $f_{t_0}$ for some particular $t_0 \in [-1, 1]$. Fix a generic path
$\{ f_t \}$ once and for all throughout this section (see \cite{Cerf} for the
precise definition of genericity). Its endpoints $f_{-1}$ and $f_1$ are strong
Morse functions on $M$. Moreover, the same holds for all but finitely many
points of $\{ f_t \}$. The rest of this subsection is devoted to describing what
changes may occur to a function at these points.

We first introduce a couple of definitions. One says that two strong Morse
functions $f$ and $g$ are isotopic if there exists a diffeomorphism $\phi$
(resp. $\psi$) of $\R$ (resp. $M$) isotopic to the identity s.t.
$g = \phi \circ f \circ \psi$. Roughly, isotopic functions represent the same
object from the viewpoint of Morse theory (see \Cref{rem:diff_s_t}).
Analogously, two paths $\{ f_t \}$ and $\{ g_t \}$ are said to be equivalent if
there exists an isotopy $\{ \phi_t \}: \R \times [-1, 1] \to \R$ (resp.
$\{ \psi_t \}: M \times [-1, 1] \to M$) s.t.
$g_t = \phi_t \circ f_t \circ \psi_t$ ($\phi_0 = \id, \psi_0 = \id$).

It is a folklore result that if path $\{ f_t \}$ consists only of strong Morse
functions then it is equivalent to a constant path. See \cite{ConCerf} for a
rigourous proof. The following description of changes of a strong Morse function
along a general generic path is to be understood as description of a certain
explicit representative in the equivalence class of a path in question.

We will depict paths of functions in the following manner. The Cerf diagram of a
path $\{ f_t \}$ is a subset of $[-1, 1] \times \R$ consisting of points
$(t, x)$ s.t. $x$ is a critical value of $f_t$. Topologically it is a set of
(possibly self-intersecting and non-closed) curves in the plane.

As proven in \cite{Cerf} in a generic one-parameter family there are two
possible changes of isotopy class of a strong Morse function, which we call
events. (Since there are only finitely many of them anyway, we assume for
convenience that $f_t$ is strong Morse for all $t$ except for a single value
$t = 0$.)
\begin{enumerate}[label=\arabic*)]
\item The function $f_0$ is strong, but non-Morse. This case is given by the
  local formula
  \[
    f_t(x_1, \dots, x_n) = x_1^3 \pm tx_1 + Q(x_2, \dots, x_n),
  \]
  where $(x_1, \dots, x_n)$ is some small coordinate neighborhood around the
  non-Morse critical point of $f_0$ (outside of a bit bigger neighborhood the
  function doesn't change at all) and $Q$ is a non-degenerate quadratic form. At
  the moment $t = 0$ the birth/death (depending on the sign) of two points of
  neighboring indices happens (the lower index is the index of $Q$). On a Cerf
  diagram this corresponds to a (left or right) cusp. This event is called
  birth/death event.
\item The function $f_0$ is Morse, but not strong. This happens when two
  critical values collide. Outside of small neighborhoods around two
  corresponding critical points the function doesn't change at all; moreover,
  critical points themselves don't move along the path.
  On the Cerf diagram this corresponds to a transversal self-intersection; in a
  sense a pair of critical values is swapped. The space on non-strong functions
  is sometimes called Maxwell stratum. So, we call this event a Maxwell event.
\end{enumerate}

\begin{rem}
  \label{rem:crit_bij}
  Note that in both cases there are exactly two distinguished critical points:
  either two newborn/about to die points or a couple with swapping critical
  values. We will refer to them as ``critical points involved in event''. They
  go one straight after another in the natural linear order. All the other
  critical points of $f_{-1}$ and of $f_1$ are in natural bijection. We will
  always keep this bijection in mind without mentioning it explicitly.
\end{rem}

Now we may describe a Cerf diagram a bit more precisely: it is a set of plane
arcs (smooth in the interior) whose endpoints are either at cusps or has $t$
coordinate equal to $\pm 1$. These arcs don't have vertical tangencies and may
self-intersect. To each arc an integer number is assigned, namely the index of
any critical point on it (mind that for a generic point in a path critical
points are in bijection with critical values). From the viewpoint of
\Cref{thm:main_thm_morse} each arc corresponds to a cell in the CW-complex
obtained via $f_t$. Birth/death event translates to birth/death of two cells of
neighboring dimensions s.t. one appears in the cellular differential of another
one with coefficient $\pm 1$. Note that this is exactly the building block of
the simple homotopy equivalence.

\subs{Any Bruhat number is realizable}
\label{subs:realiz}
In this subsection we prove

\PropRealizRestate*

The plan would be to construct a generic path of functions on $M$ which starts
with any strong Morse function and ends with the one satisfying the desired
property. The tools for constructing such a path were essentially developed by
Smale and restated in Morse-theoretical terms by Milnor in \cite{Mi2}, which is 
our main reference here. In the following three statements one is given a strong
Morse function $f_{-1}$ on a manifold $M$, equipped with a generic Riemannian
metric $g$. Recall from \Cref{subs:morse_comp} that $g$ gives rise to a Morse
complex, formally spanned by critical points. We write $\d_k$ for its
differential, which maps $k$-chains to $(k-1)$-chains. The following two
operations alter the function, but doesn't change the metric.

\begin{prop}
  \label{prop:make_birth}
  Given any $k \in \{ 0, \dots, \dim M - 1 \}$, one can find a generic path
  $\{ f_t \}$ which contains a single event, namely a birth event. The indices
  of newborn points are $k$ and $k+1$ and they lie in the small neighborhood of
  any regular point of $f_{-1}$ chosen in advance. Moreover, if $c_s$ and $c_t$
  are two critical points not involved in the event, then one appears in the
  Morse differential of another with the same coefficient for $f_{-1}$ and
  $f_1$.
\end{prop}

\begin{prop}
  \label{prop:swap}
  Suppose that $c_s$ and $c_{s-1}$ are two neighboring critical points of
  $f_{-1}$ of the same index $k$. Then one can find a generic path $\{ f_t \}$
  which contains a single event, namely a Maxwell event; its swapping points are
  $c_s$ and $c_{s-1}$. Moreover, the matrix of Morse differential $\d_k$ (resp.
  $\d_{k+1}$) for $f_1$ equals to that for $f_{-1}$ multiplied on the right
  (resp. left) by a transposition $(s, s-1)$.
\end{prop}

The next operation is called handle sliding. It alters the metric, but doesn't
change the function.

\begin{prop}
  \label{prop:handle_slide}
  Suppose that $c_s$ and $c_{s-1}$ are two neighboring critical points of
  $f_{-1}$ of the same index $k$, $k \in \{ 2, \dots, \dim M - 2 \}$. Suppose
  also that points $c_s$ and $c_{s-1}$ lie in the same connected component of
  $f_{-1}^{-1}([f_{-1}(c_{s-1}) - \eps, f_{-1}(c_s) + \eps])$. Then one can find
  a new metrc $g'_\pm$ s.t. new matrix of Morse differential $\d_k'$ (resp.
  $\d_{k+1}'$) equals to the old one $\d_k$ (resp. $\d_{k+1}$) multiplied on the
  right (resp. left) by the base change mapping $c_s$ to $c_s \pm c_{s-1}$.
\end{prop}

\begin{proof}[Proof of \Cref{prop:realiz}]
  Take any oriented strong Morse function $f_0$ on $M$ and any Riemannian metric
  $g_0$. We now begin changing $f_0$ (in the small neighborhood of its regular
  point) and $g_0$. After any birth event we orient two newborn points s.t. the
  coefficient of one in the differential of another would be $+1$ (there are two
  such choices). First of all, fix any $k \in \{ 2, \dots, \dim M - 2 \}$ and
  any $l \in \Z \setminus {0}$ s.t. the class of $l$ in $\F$ is $\lambda$.

  \begin{enumerate}[label=\arabic*.]
  \item Apply \Cref{prop:make_birth} to introduce two critical points $a$ and $b$ of
  indices $k$ and $k-1$ respecitively; call the result $f_1$.

  \item Apply it once again to introduce two new points $c$ and $d$ again of indices
  $k$ and $k-1$, s.t. $a$ and $c$ satisfy hypothesis of
  \Cref{prop:handle_slide}. The submatrix of Morse differential which takes into
  account only mentioned four points now looks like $
  \begin{pmatrix}
    p & 1 \\
    1 & q
  \end{pmatrix}
  $ for some integers $p$ and $q$. Although we don't need it, it follows from
  the \mbox{equality
    $\H_*(f_1^{-1}((-\infty, f(a) + \eps]), (-\infty, f(b) - \eps]); \Z) = 0$}
  that these numbers must satisfy $pq - 1 = \pm 1$.

  \item Apply \Cref{prop:handle_slide} to $a$ and $c$ sufficiently many times so as to
  obtain the submatrix of differential which looks like $
  \begin{pmatrix}
    p' & 1 \\
    1 & l
  \end{pmatrix}
  $ for some $p' \in \Z$. Call the new metric $g_1$.

  \item  Apply \Cref{prop:swap} to $a$ and $c$ to finally get the submatrix $
  \begin{pmatrix}
    1 & p' \\
    l & 1
  \end{pmatrix}
  $. Call the result $f_2$. It then follows straightforwardly that $a$ and $d$
  form a Barannikov pair (i.e. the bifurcation was non-trivial) with Bruhat
  number $\lambda$. Note that points $c$ and $b$ also in pair and the
  corresponding Bruhat number is $-\sfrac{1}{\lambda}$.
  \end{enumerate}
\end{proof}

\begin{rem}
  \label{rem:s4}
  Note that after performing the set of operations from the proof of
  \Cref{prop:realiz} the alternating product of Bruhat numbers doesn't change
  (up to sign; see \Cref{subs:tors}). It is not surprising since indeed it
  sometimes doesn't depend on the function at all. However, we will now sketch
  the construction of a strong Morse function on $\mathbb C\P^2$ s.t. it has only one
  Barannikov pair and the corresponding Bruhat number is $\lambda$ for any
  $\lambda$ from either $\Q$ or $\F_p$.

  Consider first a standart strong Morse function on $\mathbb C\P^2$ which has 3
  critical points $a$, $b$ and $c$ (of indices 0, 2 and 4 respectively). Apply
  \Cref{prop:make_birth} to introduce a pair of points $d$ and $e$ (of indices 2
  and 1) between (in the sense of natural linear order) $a$ and $b$. Apply
  \Cref{prop:handle_slide} enough many times s.t. $e$ would appear in the
  differential of $b$ with coeffient $l$, where $[l] = \lambda$. Finally, apply
  \Cref{prop:swap} to points $b$ and $d$ (the bifurcation will neccesarily be
  non-trival). The resulting function will have points $b$ and $e$ paired with
  Bruhat number $\lambda$.
\end{rem}

\subs{B-data in families}
\label{subs:bd_families}
In this subsection we start describing how B-data behaves along the
generic path of functions. In \Cref{subs:self_int} we finish this description.

First of all, we will orient all the functions in the path in the following way.
Pick up a generic point on some arc of the Cerf diagram. It corresponds to a
critical point of some $f_t$; orient it. Extend this orientation by continuity
to all the critical points lying the same arc (excluding the cusps). Apply this
procedure to all the arcs. This recipe allows us to orient all the points in the
path $\{ f_t \}$ by making only finite number of binary choices, namely $2^l$
where $l$ is the number of arcs. We use the term ``orientation of an arc'' for
short.

Recall that we have to fix a field $\F$ in order to define B-data.
Next, if the path $\{ f_t \}$ consists of only strong Morse functions, then this
data stays the same for all the time, see \Cref{rem:diff_s_t} and
\Cref{subs:gens_1p}.
 
We use the term ``bifurcations'' for the description of the way B-data changes
after two events from \Cref{subs:gens_1p}. Disregarding the Bruhat numbers, this
description was presented already in \cite{Bar} (see \cite{Lau} for a different
proof). See also the paper \cite{CSEM} and pictures in the survey \cite{EH}.
Thus our job is to determine how Bruhat numbers change along the way (see
\Cref{rem:or_bar}). In the case of birth/death event we restrict ourselves to
birth one for brevity (death one is obtained from birth one by reversing the
time).

Before turning to formal statements, let us describe the basic general
properties of bifurcations. Recall that in the path $\{ f_t \}$ all functions
are strong Morse except for $f_0$. Consider the following subset of critical
points of $f_{1}$: 
\begin{enumerate*}[label=\arabic*)]
\item two points involved in event (i.e. either two newborn points or a couple
  with swapping critical values, see \Cref{rem:crit_bij}),
\item the points paired with those, if any.
\end{enumerate*}
We will refer to points from this subset as ones ``involved in the
bifurcation''. The reason behind this name is that all the other points retain
their original pairs (in a sense of bijection from \Cref{rem:crit_bij}).
Moreover, the Bruhat numbers on these pairs remain the same. Thanks to all these
facts, we are able to use pictorial format for describing the bifurcations.
Namely, we will only depict points which are involved in the bifurcation (there
are at most four of them).

\begin{prop}
  \label{prop:birth_death}
  After the birth event a Barannikov pair of two newborn critical points
  appears; its Bruhat number is $\pm 1$. All the other pairs and Bruhat numbers
  remain unaltered. See \Cref{fig:birth_death}.
\end{prop}
\begin{proof}
  Recall that by the preceeding discussion in the present subsection it suffices
  to track down only the Bruhat numbers. The first statement follows from the
  description given in \Cref{subs:gens_1p}.

  Note that $\# \Cr(f_1) = \# \Cr(f_{-1}) + 2$; denote the non-Morse critical
  point of $f_0$ by $p$. Recall that outside of a small neighborhood of $p$ the
  function $f_t$ doesn't change along the path. Thus all critical points of
  $f_{-1}$ are also critical for $f_1$. Let $c_{s+1}$ and $c_s$ be two newborn
  critical points of $f_1$. Let $r_0 < \ldots < r_{\# \Cr(f_1)}$ be regular
  values of $f_1$. The sequence
  $r_0 < \ldots < r_{s-1} < r_{s+2} < r_{\# \Cr(f_1)}$ is the set of regular
  values of $f_{-1}$, so any sublevel set of $f_{-1}$ is also a sublevel set for
  $f_1$. One then arrives at the following diagram.
  \[
  \begin{tikzcd} %
    M^{r_0}     \arrow[r, symbol=\subset] \arrow[d, equal] &
    \ldots      \arrow[r, symbol=\subset]                  &
    M^{r_{s-1}} \arrow[drr, "\congrot"]   \arrow[d, equal] &
    \phantom{M^{r_s}}                     \arrow[r, symbol=\subset]&
    \phantom{M^{r_{s+1}}}                                  &
    M^{r_{s+2}} \arrow[r, symbol=\subset] \arrow[d, equal] &
    \ldots \arrow[r, symbol=\subset]                       &
    M^{r_{\# \Cr(f_1)}} \arrow[d, equal] \\
    M^{r_0}     \arrow[r, symbol=\subset] &
    \ldots      \arrow[r, symbol=\subset] &
    M^{r_{s-1}} \arrow[r, symbol=\subset] &
    M^{r_s}     \arrow[r, symbol=\subset] &
    M^{r_{s+1}} \arrow[r, symbol=\subset] &
    M^{r_{s+2}} \arrow[r, symbol=\subset] &
    \ldots      \arrow[r, symbol=\subset] &
    M^{r_{\# \Cr(f_1)}} \\
  \end{tikzcd}
\]
Here equality sign denotes set-theoretical equality of two subspaces of $M$ and
$\cong$ denotes a homeomorphism. The latter takes place since attaching two
cells as in \Cref{subs:gens_1p} doesn't change the homeomorphism type of a
space. It now follows that \Cref{con:bd_morse} produces the same B-data
for $f_{-1}$ and $f_1$ except for the above-mentioned newborn pair.
\end{proof}

\begin{rem}
  The sign in $\pm 1$ depends on the chosen orientations of arcs (see beginning
  of this subsection). See \Cref{subs:akh} for a theorem where it plays
  important role.
\end{rem}  
 
\begin{figure}[ht]
  \includegraphics[width=6cm]{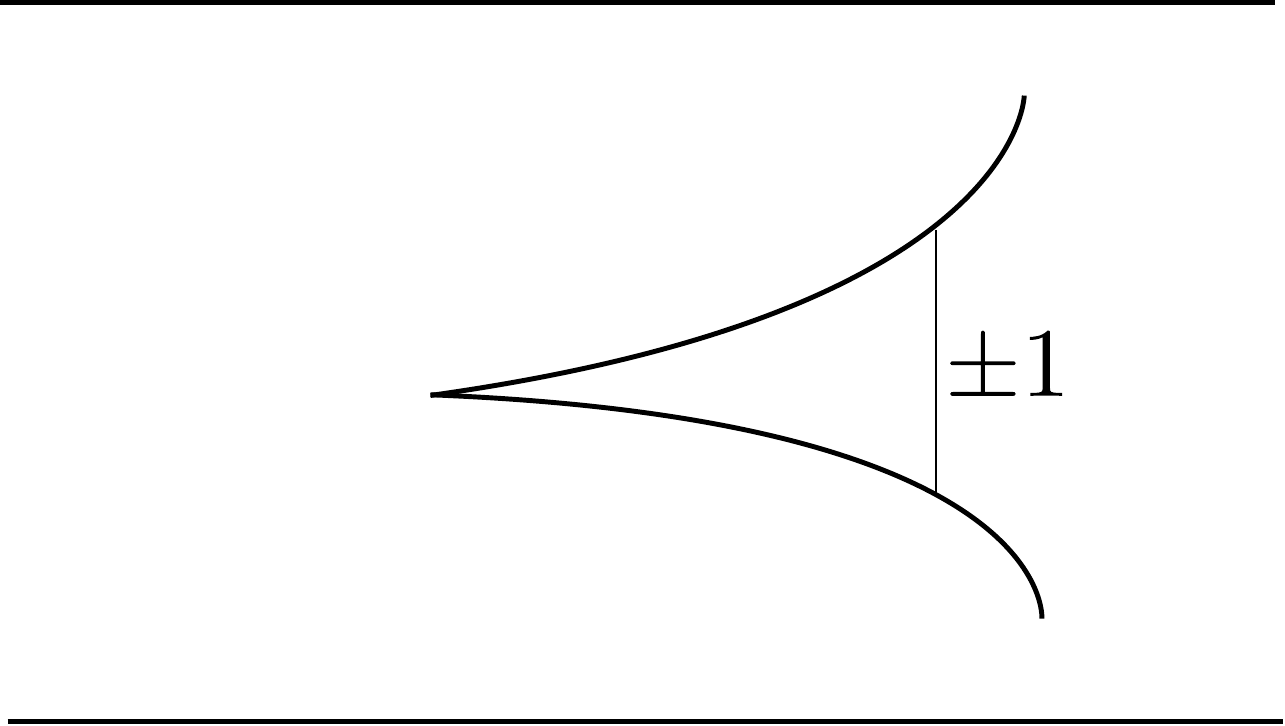}
  \caption{Birth of two critical points.}
  \label{fig:birth_death} 
\end{figure} 

\subs{Maxwell event}
\label{subs:self_int}
In this subsection we consider the second type of event, namely
self-intersection of a Cerf diagram (in other words, Maxwell event). This
finishes the description of bifurcations of B-data in families started in
\Cref{subs:bd_families}.

Let us fix the notations first. Let $c_{s+1}$ and $c_s$ be two critical points
of $f_{-1}$ participating in the bifurcation. Recall from \Cref{subs:gens_1p}
that $\Cr(f_1)$ coincides, as an ordered subset of $M$, with $\Cr(f_{-1})$ with
the order of $c_{s+1}$ and $c_s$ reversed. Let
$r_0 < \ldots < r_{\# \Cr(f_{-1})}$ be regular values of $f_{-1}$ and $f_1$.

\begin{prop}
  \label{prop:self_int}
  After the Maxwell event two types of bifurcations possible.
  \begin{enumerate}[label=\arabic*)]
  \item Trivial bifurcation. After it points $c_{s+1}$ and $c_s$ keep their
    initial pairs (if any) and Bruhat numbers on them. All combinatorial
    variants of juxtaposition of points are possible. The values $\deg c_{s+1}$
    and $\deg c_s$ may be any.
  \item Non-trivial bifurcation. The neccecesary condition is
    $\deg c_{s+1} = \deg c_s$. The list of five possible variants is given in
    \Cref{fig:non_triv_bif}.
  \end{enumerate}
  All the points not participating in the bifurcation keep their initial pairs
  (if any) and Bruhat numbers on them.
\end{prop}

\begin{figure}[ht]
  \includegraphics[width=3.5cm]{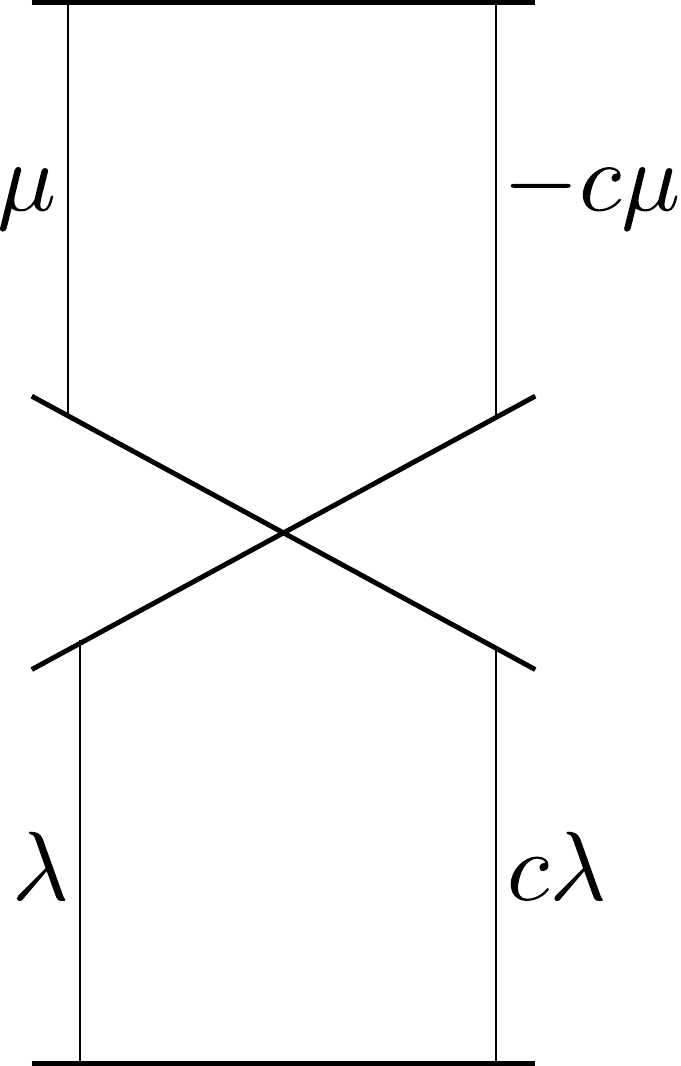}
  \includegraphics[width=3.3cm]{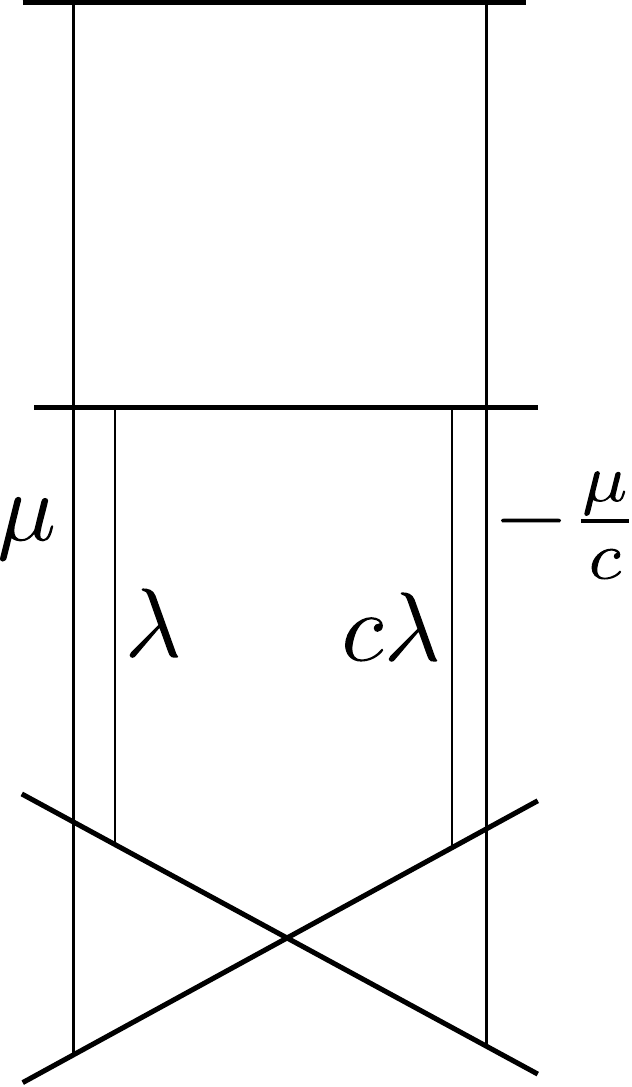}
  \includegraphics[width=3.2cm]{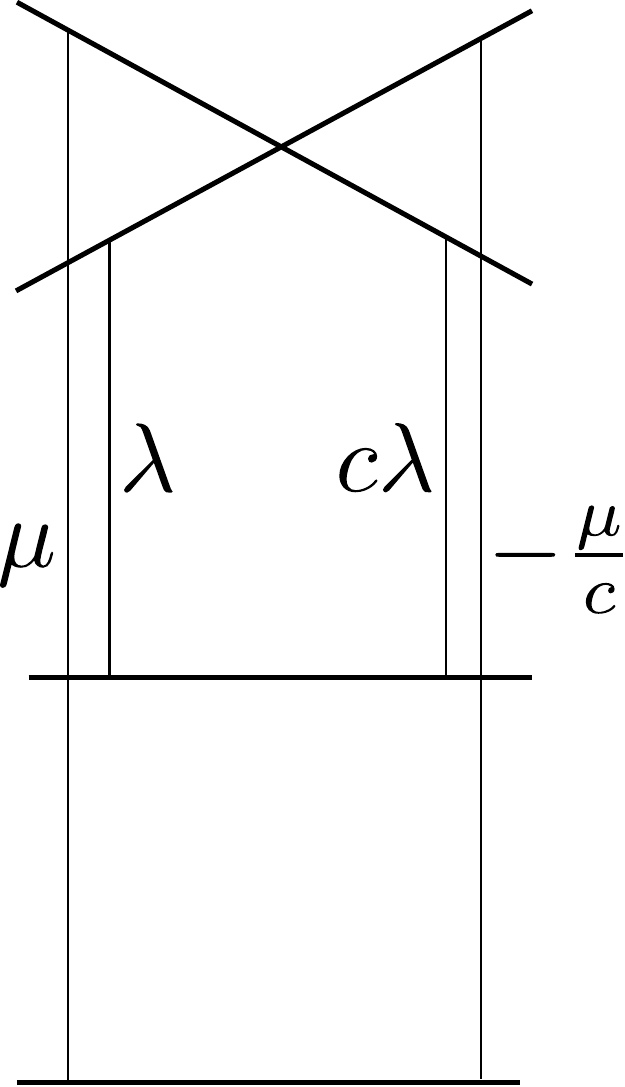}
  \includegraphics[width=3cm]{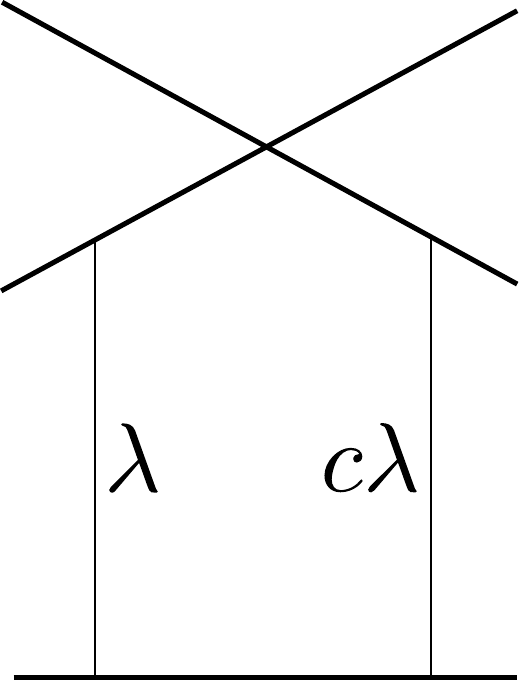} 
  \includegraphics[width=3.1cm]{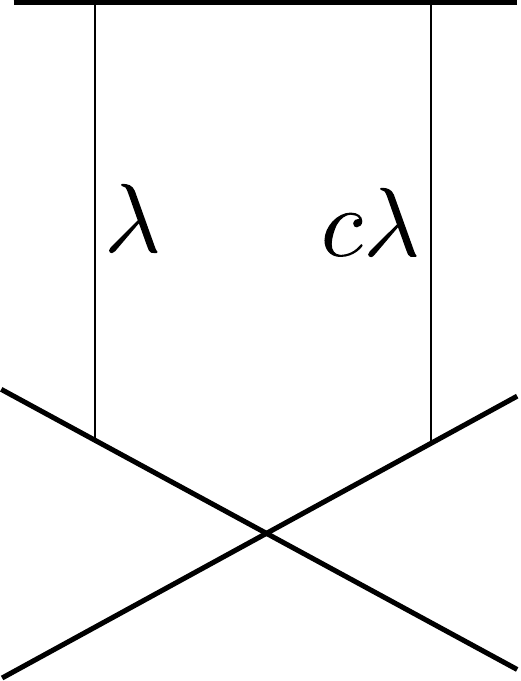}
  \caption{Non-trivial bifurations at the self-intersection of a Cerf diagram}
  \label{fig:non_triv_bif} 
\end{figure} 

\begin{rem}
  As seen on \Cref{fig:non_triv_bif} pairing and Bruhat numbers may well change
  after the non-trival bifurcation. Note that the mentioned list of five cases
  doesn't exhaust all possible combinatorial variants of juxtaposition of
  points.
\end{rem}

\begin{proof}[Proof of \Cref{prop:self_int}]
  As in \Cref{prop:birth_death} it suffices to track down only Bruhat numbers. 
  Let $r_0 < \ldots < r_{\# \Cr(f_1)}$ be regular values of $f_{-1}$. These
  values are also regular for $f_1$. Moreover, all but one sublevel sets of
  $f_{-1}$ and $f_1$ coincide. The only exception is the sublevel set for the
  regular value $r_s$; so, we write
  $M^{r_s} = \{x \in M | f_{-1}(x) \leq r_s \}$ and
  $\tilde{M^{r_s}} = \{x \in M | f_1(x) \leq r_s \}$. One arrives at the
  following diagram.
  \[
  \begin{tikzcd}
    &
    &
    &
    M^{r_s}     \arrow[dr, symbol=\subset] &
    &
    &
    \\
    M^{r_0}     \arrow[r,  symbol=\subset] &
    \ldots      \arrow[r,  symbol=\subset] &
    M^{r_{s-1}} \arrow[ur, symbol=\subset] \arrow[dr, symbol=\subset] &
                                           &
    M^{r_{s+1}} \arrow[r,  symbol=\subset] &
    \ldots      \arrow[r,  symbol=\subset] &
    M^{r_{\# \Cr(f_1)}}                   \\
    &
    &
    &
    \tilde{M^{r_s}} \arrow[ur, symbol=\subset] &
    &
    &
    \\
  \end{tikzcd}
\]
The case of trivial bifurcation and the very last statement of
\Cref{prop:self_int} now follow directly by unwrapping \Cref{con:bd_morse}. We
now need to show how Bruhat numbers change after cases 1-3 of non-trivial
bifurcation (see \Cref{fig:non_triv_bif}; two other don't place any restriction
on these numbers). By the last statement it suffices to prove that
$\tau(f_{-1}) = -\tau(f_1)$ (see \Cref{subs:tors}). Since the number $\tau$
admits torsion-theoretic interpretation, it depends only on two things: 
equivalence class basis $c$ of an enhanced complex and equivalence class of
basis $h$ of homology. The latter is uniquely determined by the enhancement on
$\H_\bullet(M)$ induced by a function. Since both points involved in the event
are not homological, this enhancement is the same for $f_{-1}$ and $f_1$.

We will now investigate how basis $c$ changes after the bifurcation. Choose any
CW-approximation (or any Riemannian metric) for $f_{-1}$ in the sense of
\Cref{subs:cw_comp_morse} (or, respectively, \Cref{subs:morse_comp}). The
resulting CW-complex (or, respectively, Morse complex) also serves as a
CW-complex associated with $f_1$ with the only difference that $s^\text{th}$ and
$(s+1)^\text{th}$ generators got swaped in the linear order. The determinant of
a transposition matrix is $-1$ and the formula $\tau(f_{-1}) = -\tau(f_1)$ now
follows from \Cref{rem:tors_basis}.
\end{proof}

\subs{Manifolds with almost no $\F$-homology}
\label{subs:no_hom}
In this subsection we prove the following theorem, see \Cref{subs:tors} for a
context.

\ThmNoHomRestate*   
                      
\begin{proof}
  Take two strong Morse functions on $M$ and connect them by a generic path. The
  plan is to use \Cref{prop:birth_death,prop:self_int} to prove that $\pm \tau$
  doesn't change after any bifurcation along the path. The case of birth/death
  event obviously follows from \Cref{prop:birth_death}. The rest is devoted to
  dealing with the Maxwell event.

  It follows from \Cref{prop:self_int} that if both points involved in Maxwell
  event are paired (i.e. not homological) then the number $\pm \tau$ stays the
  same after the bifurcation. On the other hand, by assumption any homological
  point is of degree either zero or $\dim M$. In the former case, the first part
  of \Cref{prop:pair_ext} implies that $\pm \tau$ again stays the same
  (regardless of whether the bifurcation is trivial or not). The rest is devoted
  to the latter case.

  If $M$ is non-orientable then $f$ may have a homological point of degree
  $\dim M$ only if $\ch \F = 2$, but this case makes the initial statement
  trivial (see \Cref{rem:q_fp}). If $M$ is orientable we make use of the second
  part of \Cref{prop:pair_ext} in the same way as earlier.
\end{proof}

\subs{A theorem of Akhmetev-Cencelj-Repovs} \label{subs:akh} In this subsection
we apply our methods to reprove the theorem of Akhmetev-Cencelj-Repovs
\cite{Akhm} in greater generality. Roughly it says that several numerical
invarinats of a generic path in the space of strong Morse functions satisfy a
certain equation mod 2.

First of all we need to pass to a bit more general setting. Cobordism is a
manifold $M$ with boundary $\dM_0 \sqcup \dM_1$. By a Morse function $f$ on a
cobordism $(M, \dM_0, \dM_1)$ we will mean a function $f \colon M \to [0, 1]$
with only non-degenerate critical points, all in the interior $M$, s.t.
$f^{-1}(0) = \dM_0$ and $f^{-1}(1) = \dM_1$. Strongness property is defined in
the same manner as in the closed case. All the classical results from
\Cref{subs:cw_comp_morse,subs:morse_comp} generalize readily to this setting,
see \cite{Mi2}. The only thing to mention here is that now construction from
\Cref{subs:cw_comp_morse} attaches cells one-by-one starting from $\dM_0$. As a
consequence, one obtains a CW-decomposition of $\sfrac{M}{\dM_0}$, not $M$
itself. Thus, all the various constructed complexes count relative homology
$\H_*(M, \dM_0)$. All the results about enhanced complexes also translate
readily. Trivial cobordism is a cylinder
$(M \times [0, 1], M \times \{ 0 \}, M \times \{ 1 \})$.

We will now introduce mentioned invariants of a generic path $\{ f_t \}$. The
first one is the number of self-intersections of the Cerf diagram (or, in our
terminology, the number of Maxwell events), call it $\X$. To get to the second
one recall that in \Cref{subs:bd_families} we described the procedure of
orienting the arcs of a Cerf diagram, which outputs an orientation of each
strong Morse function in a path. After orienting the arcs somehow one can assign
a sign to each cusp of a Cerf diagram (i.e. birth/death event) as follows. Let
$t_0$ be a point of birth (resp. death) event. Pick any value $t_1 > t_0$ (resp.
$t_1 < t_0$) s.t. all functions between $t_1$ and $t_0$ ($t_0$ excluded) are
strong Morse. Denote by $c_{s+1}$ and $c_s$ two newborn (resp. about to die)
critical points of $f_{t_1}$. It follows from classical results recalled in
\Cref{subs:cw_comp_morse} that cellular differential of $c_{s+1}$ contains $c_s$
with coefficient either 1 or $-1$ regardless of choices made (essentially,
cellular approximation). Using another language, one may say the same about the
Morse differential w.r.t. to any Riemannian metric (see \Cref{subs:morse_comp}).
The sign of a cusp is now defined as the sign of this number. Let $\nC$ be the
number of negative cusps. Changing orientation of an arc changes the sign of
each cusp which serves as this arc's endpoint (obviously, there are at most two
such cusps). Therefore, if both $f_{-1}$ and $f_1$ have no critical points, then
the parity of $\nC$ is a well-defined invariant of a path $\{ f_t \}$. We are
now turning to a corollary which is easier to state compared to the main
theorem. For example, it doesn't appeal to any field or Bruhat numbers
whatsover, thus remaining entirely in the realm of Cerf theory.

\CorAkhCorRestate*

\begin{rem}
  In \cite{Akhm} \Cref{cor:akh_cor} was proved using two different methods, both
  requiring additional assumptions on $M$. The first method is based on
  h-principle and requires $M$ to be stably parallelizable and simply-connected.
  The second one is based on parametric Morse theory and requires the dimension
  of $M$ to be at least five. This method is geometrical and relies on strong
  results of Smale akin to those in \Cref{subs:realiz}. These results, in turn,
  only valid if the dimension of $M$ is big enough and lead, for example, to the
  famous h-cobordism theorem. On the other hand, our approach of using enhanced
  complexes, after deducing all the necessary generalities in
  \Cref{subs:bd_families}, is entirely combinatorial. 
\end{rem}

We will now give one more definition in order to state the main theorem of this
subsection. Fix a field $\F$ s.t. $\ch \F \neq 2$ and consider the set of
Barannikov pairs of some oriented strong Morse function $f$. We say that two
pairs overlap if they overlap when viewed as segments on the real line.
Formally, two pairs $(s_1, s_2)$ and $(t_1, t_2)$ overlap if either
$s_1 < t_1 < s_2 < t_2$ or $t_1 < s_1 < t_2 < s_2$. Let $\O$ be the number of
all overlapping (unordered) pairs of Barannikov pairs (we stress out that we
place no restrictions on the degrees of points). Define now
\[
  \tau'(f, \F) \colonequals (-1)^{\O} \prod_{s \in U}\lambda(s)^{(-1)^{\deg s}}
  \in \F^*.
\]
\begin{rem}
  In words, this is the alternating product of all the Bruhat numbers times the
  sign depending on the parity of $\O$. This sign is different from the one from
  \Cref{def:tors}. As usual, we drop the ingredients of $\tau'$ when they are
  understood.
\end{rem}

Let now $(M, \dM_1, \dM_2)$ be any cobordism s.t. relative homology
$\H_*(M, \dM_1)$ vanishes (recall that we take coefficients to be in $\F$ by
default). (For example, one may take an h-cobordism.) Then \Cref{thm:no_hom} (in
the aforementioned setting) implies that $\pm \tau'(f)$ in independent of $f$.

We are now ready to state the main theorem of this subsection, which we do in
multiplicative notation. Recall that $\X$ is the number of self-intersections of
the Cerf diagram and $\nC$ is the number of its negative cusps.

\begin{thm}
  \label{thm:akh}
  Let $\F$ be a field and $(M, \dM_1, \dM_2)$ be a cobordism s.t.
  $\H_*(M, \dM_1) = 0$. Let also $\{ f_t \}$ be a (somehow oriented) generic
  path of functions on it. Then one has
  \[
    \frac{\tau'(f_1)}{\tau'(f_{-1})} (-1)^{\X} (-1)^{\nC} = 1.
  \]
\end{thm}
\begin{proof}
  The plan is to track down when the number $\tau'(f_t)$ changes its sign as $t$
  varies from $-1$ to 1. It suffices to prove that it does so exactly after $t$
  passes \begin{enumerate*}[label=\arabic*)] \item a self-intersection of the
    Cerf diagram,\item a negative cusp (either left or right)\end{enumerate*}.
  Note that by assumption there are no homological points, i.e. all of the
  points are paired.

  We will first sort out the second case. The number $\O$ doesn't change after a
  birth/death event, since the newborn/recently dead Barannikov pair doesn't
  overlap with any other pair. The statement now follows directly from
  \Cref{prop:birth_death}.

  We will now turn to the first case. Suppose that the bifurcation is
  non-trivial. Then the decomposition of critical points into pairs remains the
  same, thus the number $\O$ remains unaltered. It then follows directly from
  \Cref{prop:birth_death} that the alternating product of Bruhat numbers change
  its sign. Suppose now that the bifurcation is trivial. This time the set of
  Bruhat numbers remains the same, but the number $\O$ increases or decreases by
  one (depending on the justaposition of pairs involved in the bifurcation).
\end{proof}

Since $\tau(f, \F) = 1$ for any $f$ without critical points at all (and for any
$\F$), the \Cref{cor:akh_cor} follows straightforwardly by taking any $\F$ s.t.
$\ch \F \neq 2$.

\clearpage

\printbibliography           %

\end{document}